\documentclass[12pt,a4paper,twoside]{article}

\title{\sc Regularity Results for Generalized Electro-Magnetic Problems}
\def\shorttitle{Regularity Results for Generalized Electro-Magnetic Problems}
\def\pauthor{Peter Kuhn and Dirk Pauly}

\def\mylabelonoff{off}
\def\allowdisbrk{no}

\usepackage{a4,exscale,ifthen,amsfonts,amssymb,amsmath,amscd,graphicx}
\usepackage[english]{babel}
\usepackage[mathscr]{eucal}

\author{{\sf\pauthor}}
\numberwithin{equation}{section}

\newenvironment{acknow}{{\vspace*{1cm}\noindent\bf Acknowledgements }}{}
\newcommand{\bewboxw}{\mbox{}\hfill $\square$ \\}

\newenvironment{proof}{{\noindent\bf Proof }}{\bewboxw}

\newcommand{\keywords}[1]{{\noindent\bf Key Words }#1}
\newcommand{\amsclass}[1]{{\noindent\bf AMS MSC-Classifications }#1}

\ifthenelse{\equal{\mylabelonoff}{on}}
{\newcommand{\mylabel}[1]{\label{#1}\fbox{{\rm #1}}}}{\newcommand{\mylabel}[1]{\label{#1}\makebox[0mm][]{}}}
\ifthenelse{\equal{\allowdisbrk}{yes}}
{\allowdisplaybreaks}{}

\newcommand{\paper}[7]{\bibitem{#1} #2, `#3', {\it #4}, #5, (#6), #7.}

\newcommand{\book}[6]{\bibitem{#1} #2, {\it #3}, #4, #5, (#6).}

\newcommand{\dissavail}[6]{\bibitem{#1} #2, `#3', {\sf Dissertation}, #4, (#5), available from {\tt #6}.}
\newcommand{\habil}[5]{\bibitem{#1} #2, `#3', {\sf Habilitationsschrift}, #4, (#5).}

\usepackage{amsfonts,amssymb,amsmath,amscd}
\usepackage[mathscr]{eucal}

\newcommand{\schluss}{\ifodd\value{page}\newpage\thispagestyle{empty}\makebox[0mm][]{}\color{sehrhell}.\fi\end{document}}

\newcommand{\normabst}{\hspace{-0.4ex}}

\newcommand{\ol}[1]{\overline{#1}}

\newcommand{\ub}{\underbrace}
\newcommand{\qtext}[1]{\quad\text{#1}\quad}
\newcommand{\qqtext}[1]{\qquad\text{#1}\qquad}

\newcommand{\me}{{-1}}

\newcommand{\intom}{\int_\Omega}

\newcommand{\om}{{\Omega}}

\newcommand{\omq}{\ol{\om}}
\newcommand{\omb}{\om_\mathrm{b}}

\newcommand{\dom}{{\p\om}}

\newcommand{\ohne}{\setminus}

\newcommand{\eps}{\varepsilon}

\newcommand{\epsd}{\hat{\eps}}

\newcommand{\EH}{(E,H)}

\newcommand{\To}{\longrightarrow}

\newcommand{\Equi}{\Longleftrightarrow}

\newcommand{\restr}[2]{\left.#1 \right|_{#2}}

\newcommand{\Abb}[5]{\begin{array}{ccccc}#1&:&#2&\longrightarrow&#3\\{}&{}&#4&\longmapsto&#5\end{array}}

\newcommand{\bds}{\begin{displaystyle}}
\newcommand{\eds}{\end{displaystyle}}
\newcommand{\beq}{\begin{equation}}
\newcommand{\eeq}{\end{equation}}
\newcommand{\beqa}{\begin{eqnarray}}
\newcommand{\eeqa}{\end{eqnarray}}
\newcommand{\beqanon}{\begin{eqnarray*}}
\newcommand{\eeqanon}{\end{eqnarray*}}
\newcommand{\bamath}{$$\begin{array}}
\newcommand{\eamath}{\end{array}$$}
\newcommand{\ba}{\begin{array}}
\newcommand{\ea}{\end{array}}
\newcommand{\bc}{\begin{center}}
\newcommand{\ec}{\end{center}}
\newcommand{\bmini}{\begin{minipage}}
\newcommand{\emini}{\end{minipage}}

\newtheorem{lem}{Lemma}[section]
\newtheorem{theo}[lem]{Theorem}
\newtheorem{cor}[lem]{Corollary}
\newtheorem{defini}[lem]{Definition}
\newtheorem{rem}[lem]{Remark}

\newcommand{\set}[2]{\{#1 \;\text{\bf :}\;  #2\}}

\newcommand{\setb}[2]{\big\{#1 \;\text{\bf :}\;  #2\big\}}

\newcommand{\rz}{\mathbb{R}}

\newcommand{\nz}{\mathbb{N}}

\newcommand{\cz}{\mathbb{C}}

\newcommand{\nzn}{{\nz_0}}
\newcommand{\rd}{{\rz^3}}
\newcommand{\rN}{{\rz^N}}

\newcommand{\rzp}{{\rz_+}}

\newcommand{\calF}{{\mathscr{F}}}

\newcommand{\calO}{{\mathscr{O}}}

\newcommand{\calH}{{\mathscr{H}}}

\newcommand{\calM}{{\mathscr{M}}}

\DeclareMathOperator{\p}{\partial}
\DeclareMathOperator{\pa}{\p^\alpha}

\DeclareMathOperator{\loc}{loc}

\DeclareMathOperator{\MAX}{{{\calM}{\rm ax}}}

\DeclareMathOperator{\id}{Id}

\DeclareMathOperator{\supp}{supp}

\DeclareMathOperator{\rot}{rot}
\DeclareMathOperator{\Rot}{Rot}

\DeclareMathOperator{\pdiv}{div}

\DeclareMathOperator{\ie}{i}

\DeclareMathOperator{\pd}{d\hspace{-0.4ex}}

\DeclareMathOperator{\curl}{curl}

\newcommand{\pcoverset}[2]{\overset{#2}{\mathrm{C}}{}^{#1}}
\newcommand{\pc}[1]{\pcoverset{#1}{}}
\newcommand{\cn}[1]{\pcoverset{#1}{\circ}}
\newcommand{\cu}{\pc{\infty}}

\newcommand{\com}[1]{\pc{#1}(\Omega)}

\newcommand{\lz}{\mathrm{L}^2}
\newcommand{\lzom}{\lz(\Omega)}
\newcommand{\Lz}[1]{\lz_{#1}}

\newcommand{\Lzsom}{\Lz{s}(\Omega)}

\newcommand{\Lzloc}{\Lz{\loc}}
\newcommand{\Lzlocom}{\Lzloc(\Omega)}

\newcommand{\h}[3]{\overset{#3}{\mathrm{H}}{}^{#1}_{#2}}
\renewcommand{\H}[3]{\overset{#3}{\mathbf{H}}{}^{#1}_{#2}}
\newcommand{\hon}[2]{\h{#1}{#2}{\circ}}

\newcommand{\hm}{\h{m}{}{}}

\newcommand{\hms}{\h{m}{s}{}}
\newcommand{\Hms}{\H{m}{s}{}}

\renewcommand{\hom}[2]{\h{#1}{#2}{}(\Omega)}
\newcommand{\Hom}[2]{\H{#1}{#2}{}(\Omega)}

\newcommand{\hmom}{\hm(\Omega)}

\newcommand{\hmsom}{\hms(\Omega)}
\newcommand{\Hmsom}{\Hms(\Omega)}

\newcommand{\qc}[3]{\overset{#3}{\mathrm{C}}{}^{#1,#2}}
\newcommand{\cq}[1]{\qc{#1}{q}{}}
\newcommand{\cqn}[1]{\qc{#1}{q}{\circ}}
\newcommand{\cqu}{\cq{\infty}}
\newcommand{\cqun}{\cqn{\infty}}

\newcommand{\cquom}{\cqu(\Omega)}
\newcommand{\cqunom}{\cqun(\Omega)}

\newcommand{\qlz}[1]{\mathrm{L}^{2,#1}}
\newcommand{\qlzom}[1]{\qlz{#1}(\Omega)}
\newcommand{\lzq}{\qlz{q}}
\newcommand{\lzqom}{\qlzom{q}}
\newcommand{\qLz}[2]{\qlz{#1}_{#2}}
\newcommand{\qLzom}[2]{\qLz{#1}{#2}(\Omega)}
\newcommand{\Lzq}[1]{\qLz{q}{#1}}
\newcommand{\Lzqom}[1]{\Lzq{#1}(\Omega)}
\newcommand{\qLzs}[1]{\qLz{#1}{s}}

\newcommand{\qLzt}[1]{\qLz{#1}{t}}

\newcommand{\Lzqs}{\qLzs{q}}
\newcommand{\Lzqsom}{\Lzqs(\Omega)}
\newcommand{\Lzqt}{\qLzt{q}}
\newcommand{\Lzqtom}{\Lzqt(\Omega)}

\newcommand{\Lzqloc}{\Lzq{\loc}}

\newcommand{\lzqpe}{\qlz{q+1}}
\newcommand{\lzqpeom}{\qlzom{q+1}}
\newcommand{\Lzqpe}[1]{\qLz{q+1}{#1}}
\newcommand{\Lzqpeom}[1]{\Lzqpe{#1}(\Omega)}
\newcommand{\Lzqpes}{\qLzs{q+1}}
\newcommand{\Lzqpesom}{\Lzqpes(\Omega)}

\newcommand{\qh}[4]{\overset{#4}{\mathrm{H}}{}^{#1,#2}_{#3}}
\newcommand{\qH}[4]{\overset{#4}{\mathbf{H}}{}^{#1,#2}_{#3}}

\newcommand{\hqm}{\qh{m}{q}{}{}}
\newcommand{\Hqm}{\qH{m}{q}{}{}}
\newcommand{\hqms}{\qh{m}{q}{s}{}}
\newcommand{\Hqms}{\qH{m}{q}{s}{}}

\newcommand{\qhom}[4]{\qh{#1}{#2}{#3}{#4}(\Omega)}
\newcommand{\qHom}[4]{\qH{#1}{#2}{#3}{#4}(\Omega)}

\newcommand{\qHonom}[3]{\qHom{#1}{#2}{#3}{\circ}}

\newcommand{\Hqmom}{\qHom{m}{q}{}{}}
\newcommand{\hqmsom}{\qhom{m}{q}{s}{}}
\newcommand{\Hqmsom}{\qHom{m}{q}{s}{}}

\newcommand{\hqmlocom}{\qhom{m}{q}{\loc}{}}

\newcommand{\Hqonmom}{\qHonom{m}{q}{}}

\newcommand{\pr}[4]{{}_{#3}\overset{#4}{\mathrm{R}}{}^{#1}_{#2}}
\newcommand{\prq}[1]{\pr{q}{#1}{}{}}

\newcommand{\rqs}{\pr{q}{s}{}{}}

\newcommand{\rqn}[1]{\pr{q}{#1}{0}{}}

\newcommand{\ronq}[1]{\pr{q}{#1}{}{\circ}}

\newcommand{\ronqs}{\pr{q}{s}{}{\circ}}

\newcommand{\ronqn}[1]{\pr{q}{#1}{0}{\circ}}

\newcommand{\ronqsn}{\pr{q}{s}{0}{\circ}}

\newcommand{\rqom}[1]{\prq{#1}(\Omega)}

\newcommand{\rqsom}{\rqs(\Omega)}

\newcommand{\rqnom}[1]{\rqn{#1}(\Omega)}

\newcommand{\ronqom}[1]{\ronq{#1}(\Omega)}

\newcommand{\ronqsom}{\ronqs(\Omega)}

\newcommand{\ronqnom}[1]{\ronqn{#1}(\Omega)}

\newcommand{\ronqsnom}{\ronqsn(\Omega)}

\newcommand{\pR}[4]{{}_{#3}\overset{#4}{\mathbf{R}}{}^{#1}_{#2}}
\newcommand{\pRq}[1]{\pR{q}{#1}{}{}}

\newcommand{\Rqs}{\pR{q}{s}{}{}}

\newcommand{\Ronq}[1]{\pR{q}{#1}{}{\circ}}

\newcommand{\Ronqs}{\pR{q}{s}{}{\circ}}

\newcommand{\Ronqsn}{\pR{q}{s}{0}{\circ}}

\newcommand{\Rqom}[1]{\pRq{#1}(\Omega)}

\newcommand{\Rqsom}{\Rqs(\Omega)}

\newcommand{\Ronqom}[1]{\Ronq{#1}(\Omega)}

\newcommand{\Ronqsom}{\Ronqs(\Omega)}

\newcommand{\Ronqsnom}{\Ronqsn(\Omega)}

\newcommand{\pdi}[4]{{}_{#3}\overset{#4}{\mathrm{D}}{}^{#1}_{#2}}
\newcommand{\pdq}[1]{\pdi{q}{#1}{}{}}
\newcommand{\dqpe}[1]{\pdi{q+1}{#1}{}{}}
\newcommand{\dqs}{\pdi{q}{s}{}{}}

\newcommand{\dqn}[1]{\pdi{q}{#1}{0}{}}
\newcommand{\dqpen}[1]{\pdi{q+1}{#1}{0}{}}

\newcommand{\dqom}[1]{\pdq{#1}(\Omega)}
\newcommand{\dqpeom}[1]{\dqpe{#1}(\Omega)}
\newcommand{\dqsom}{\dqs(\Omega)}

\newcommand{\dqnom}[1]{\dqn{#1}(\Omega)}
\newcommand{\dqpenom}[1]{\dqpen{#1}(\Omega)}

\newcommand{\pDi}[4]{{}_{#3}\overset{#4}{\mathbf{D}}{}^{#1}_{#2}}
\newcommand{\pDq}[1]{\pDi{q}{#1}{}{}}
\newcommand{\Dqpe}[1]{\pDi{q+1}{#1}{}{}}
\newcommand{\Dqs}{\pDi{q}{s}{}{}}

\newcommand{\Dqn}[1]{\pDi{q}{#1}{0}{}}

\newcommand{\Dqom}[1]{\pDq{#1}(\Omega)}
\newcommand{\Dqpeom}[1]{\Dqpe{#1}(\Omega)}
\newcommand{\Dqsom}{\Dqs(\Omega)}

\newcommand{\Dqnom}[1]{\Dqn{#1}(\Omega)}

\newcommand{\dH}[3]{{}_{#3}\calH^{#1}_{#2}}
\newcommand{\dhqom}{\dH{q}{}{}(\Omega)}
\newcommand{\dhqmeom}{\dH{q}{-1}{}(\Omega)}

\newcommand{\dhqepsom}[1]{\dH{q}{#1}{\eps}(\Omega)}

\newcommand{\skp}[2]{\langle#1,#2\rangle}

\newcommand{\skpb}[2]{\big\langle#1,#2\big\rangle}

\newcommand{\skpom}[2]{\langle#1,#2\rangle_{\Omega}}

\newcommand{\norm}[1]{|\normabst|#1|\normabst|}
\newcommand{\normb}[1]{\big|\normabst\big|#1\big|\normabst\big|}

\DeclareMathOperator{\grad}{grad}

\usepackage{algorithm,algorithmic}

\newcommand{\Es}{\tilde{E}}

\def \vp{\varphi}

\def \vp{\varphi}
\def \eps{\varepsilon}

\newcommand{\xo}{\Omega}

\newcommand{\xu}{U}
\newcommand{\xolu}{\ol{\xu}}
\newcommand{\xv}{V}
\newcommand{\xolv}{\ol{\xv}}

\newcommand{\von}[1]{{(#1)}}

\newcommand{\vono}{\von{\Omega}}
\newcommand{\vonpo}{\von{\partial\Omega}}

\newcommand{\vonolo}{\von{\ol{\Omega}}}

\newcommand{\xr}[3]{{}_{#2}\overset{#3}{\mathbf{R}}{}^{#1}}
\newcommand{\xrq}{\xr{q}{}{}}

\newcommand{\xrqm}{\xr{q-1}{}{}}

\newcommand{\xrnq}{\xr{q}{}{\circ}}

\newcommand{\xrnqm}{\xr{q-1}{}{\circ}}

\newcommand{\xrnqn}{\xr{q}{0}{\circ}}
\newcommand{\xrnqpn}{\xr{q+1}{0}{\circ}}

\newcommand{\xd}[3]{{}_{#2}\overset{#3}{\mathbf{D}}{}^{#1}}
\newcommand{\xdq}{\xd{q}{}{}}
\newcommand{\xdqp}{\xd{q+1}{}{}}

\newcommand{\xdqn}{\xd{q}{0}{}}

\newcommand{\xdqmn}{\xd{q-1}{0}{}}

\newcommand{\xsc}[2]{\overset{#2}{\mathrm{C}}{}^{#1}}

\newcommand{\xc}[3]{\overset{#3}{\mathrm{C}}{}^{#1 , #2}}
\newcommand{\xciq}{\xc{\infty}{q}{}}

\newcommand{\xcinqp}{\xc{\infty}{q+1}{\circ}}

\newcommand{\xh}[3]{{\overset{#3}{\mathbf{H}}}{}^{#1,#2}}
\newcommand{\xhmq}{\xh{m}{q}{}}
\newcommand{\xhmqp}{\xh{m}{q+1}{}}

\newcommand{\xhmqm}{\xh{m}{q-1}{}}

\newcommand{\xhmpq}{\xh{m+1}{q}{}}

\newcommand{\xhmmq}{\xh{m-1}{q}{}}
\newcommand{\xhmmqp}{\xh{m-1}{q+1}{}}

\newcommand{\xhmmqm}{\xh{m-1}{q-1}{}}

\newcommand{\xl}[1]{{\mathrm{L}}{}^{2,#1}}
\newcommand{\xlq}{\xl{q}}
\newcommand{\xlqp}{\xl{q+1}}

\newcommand{\xxd}[1]{\mathcal{D}{}^{#1}}
\newcommand{\xxdq}{\xxd{q}}

\newcommand{\xrot}{\mathrm{rot}}
\newcommand{\xdiv}{\mathrm{div}}

\def \id{{\rm id}}
\def \d{{\rm d}}
\def \dx{{\rm dx}}

\def \xgt{\gamma_t}
\def \xgn{\gamma_n}

\def \xchgt{\check{\gamma}_t}
\def \xchgn{\check{\gamma}_n}

\def \normu#1#2{|\normabst| #1 |\normabst|_{ #2 }}
\def \spu#1#2#3{\langle #1 , #2 \rangle_{ #3 }}

\newcommand{\Eh}{\check{E}}
\newcommand{\pkdiv}{\mathrm{div}}
\renewcommand{\rot}{\pd}
\renewcommand{\xrot}{\rot}
\renewcommand{\pdiv}{\delta}
\renewcommand{\xdiv}{\pdiv}
\renewcommand{\Rot}{\rot_{\,\,\dom}}

\renewcommand{\pr}[4]{{}_{#3}\overset{#4}{\mathrm{D}}{}^{#1}_{#2}}
\renewcommand{\pR}[4]{{}_{#3}\overset{#4}{\mathbf{D}}{}^{#1}_{#2}}
\renewcommand{\pdi}[4]{{}_{#3}\overset{#4}{\Delta}{}^{#1}_{#2}}
\renewcommand{\pDi}[4]{{}_{#3}\overset{#4}{\mbox{\boldmath$\Delta$}}{}^{#1}_{#2}}
\renewcommand{\xr}[3]{\pr{#1}{}{#2}{#3}}
\renewcommand{\xd}[3]{\pdi{#1}{}{#2}{#3}}
\renewcommand{\xhmq}{\qh{m}{q}{}{}}
\renewcommand{\xhmqp}{\qh{m}{q+1}{}{}}
\renewcommand{\xhmqm}{\qh{m}{q-1}{}{}}
\renewcommand{\xhmpq}{\qh{m+1}{q}{}{}}

\renewcommand{\xhmmq}{\qh{m-1}{q}{}{}}
\renewcommand{\xhmmqp}{\qh{m-1}{q+1}{}{}}
\renewcommand{\xhmmqm}{\qh{m-1}{q-1}{}{}}

\begin{document}

\date{2008}
\maketitle{}

\begin{abstract}
We prove regularity results up to the boundary for time independent generalized Maxwell equations on Riemannian manifolds
with boundary using the calculus of alternating differential forms. We discuss homogeneous and inhomogeneous
boundary data and show `polynomially weighted' regularity in exterior domains as well.\\
\\
\amsclass{35Q60, 78A25, 78A30}
\end{abstract}

\tableofcontents

\section{Introduction}

Regularity theorems are important tools in almost all fields of partial differential equations.
In our efforts to completely determine the low frequency behavior 
of the time-harmonic solutions of the generalized Maxwell's equations in exterior domains of $\rN$
\cite{paulydiss,paulytimeharm,paulystatic,paulydeco,paulyasym}
as well as to prove compactness results and trace theorems for Sobolev spaces of differential forms
on $N$-dimensional Riemannian manifolds \cite{kuhndiss}
we have been forced to show regularity results, which meet our needs.
Here `generalized' means using the calculus 
of alternating differential forms on Riemannian manifolds of arbitrary dimension, 
which is a convenient and well-known way to formulate Maxwell's equations and to emphasize their 
independence of the special choice of a coordinate system.
Since these results are of particular interest of their own 
we will prove in the paper at hand results for the time independent case like the following:

{\sf Let $M$ be a $N$-dimensional smooth Riemannian manifold and $\om\subset M$ be some connected open subset.
If the exterior derivative of some differential form $E$ from $\Lzsom$ and the co-derivative of $\eps E$ belong to some
suitable weighted Sobolev space $\hom{m}{s+1}$ 
and the tangential trace $\iota^*E$ belongs to the corresponding 
trace Sobolev space $\h{m+1/2}{}{}(\dom)$ as well, 
then $E$ already belongs to the higher order Sobolev space $\hom{m+1}{s}$\,.}
(For details please see section \ref{results}.)

Here $\eps$ is a real valued, symmetric, bounded and uniformly positive definite linear transformation 
(one may think of a matrix) on differential forms, $\iota$ denotes the natural embedding of the boundary, i.e.
$\iota:\dom\hookrightarrow\omq$\,, and $s\in\rz$ indicates some polynomially weight.
For manifolds with compact closure, i.e. `bounded domains', the weight $s$ plays no role 
since then all results for $s$ are equivalent to the special case $s=0$\,. 

Regularity results as well as regularity estimates, which automatically will be shown within our proofs,
presented here are flexibly usable in the context of time independent generalized Maxwell's equations.
For example, if we consider (linear media and) the static generalized Maxwell equations
\begin{align*}
\rot E&=G&&,&\pdiv\eps E&=f&&,&\iota^*E&=\lambda&&,\\
\pdiv H&=F&&,&\rot\mu H&=g&&,&\iota^*\mu H&=\kappa
\end{align*}
or the time-harmonic generalized Maxwell equations (with frequency $\omega$)
\begin{align*}
\rot E+\ie\omega\mu H&=G&&,&\iota^*E&=\lambda&&,\\
\pdiv H+\ie\omega\eps E&=F&&,&\iota^*\mu H&=\kappa&&,
\end{align*}
e.g. arising from the full generalized Maxwell equations by Fourier's transformation with respect to time 
(or a time-harmonic ansatz), 
we get regularity of the solutions and corresponding estimates immediately or by induction, respectively.     

We should mention that the generalized Maxwell equations also comprise the system of linear acoustics 
and the 2-dimensional version of Maxwell's equations as well as periodic boundary conditions
in a unified approach. 

In the special classical case of bounded sub-domains of the Euclidian space $\rd$ and homogeneous boundary traces
such results for Maxwell problems have been proved earlier by Weber \cite{weberreg}.

\section{Preliminaries and definitions}

Let $M$ be a $N$-dimensional smooth Riemannian manifold
and $\om\subset M$ denote some connected open subset with compact closure in $M$\,.
On $\cqunom$\,, the vector space of all smooth ($\cu$) differential forms of rank $q$
(shortly $q$-forms) on $\om$ with compact support in $\om$\,, we have a scalar product
$$\skp{E}{H}_{\lzqom}:=\intom E\wedge*\bar{H}$$
and thus we may define $\lzqom$\,, the Hilbert space of all square integrable $q$-forms,
as the closure of $\cqunom$ in the corresponding induced norm. Utilizing the weak version of Stokes' theorem
$$\skp{\pd E}{H}_{\lzqpeom}=-\skp{E}{\delta H}_{\lzqom}\qquad\forall\;\EH\in\cn{\infty,q,q+1}(\om)$$
(with an obvious notation)
we can define weak versions of the exterior derivative and the co-derivative.
Hence we can introduce the Hilbert spaces (equipped with their natural graph norms)
\begin{align*}
\rqom{}&:=\setb{E\in\lzqom}{\pd E\in\lzqpeom}\qquad,\\
\dqpeom{}&:=\setb{H\in\lzqpeom}{\pdiv H\in\lzqom}
\intertext{and their closed subspaces}
\rqnom{}&:=\setb{E\in\lzqom}{\pd E=0}\qquad,\\
\dqpenom{}&:=\setb{H\in\lzqpeom}{\pdiv H=0}\qquad.
\end{align*}
Using charts we may define the usual Sobolev spaces $\hqm(\om)$ of real order $m\geq0$\,.
For this we need a finite chart family $(V_\ell.h_\ell)$\,, $\ell=1,\dots,L$\,, covering the compact set $\omq$\,.
Then we write $E\in\hqm(\om)$\,, if and only if $E_I^\ell\in\hm\big(h_\ell(V_\ell\cap\om)\big)$
for all $I$ and
$$\norm{E}_{\hqm(\om)}:=\big(\sum_{\ell=1}^{L}\sum_I\norm{E_I^\ell}_{\hm(h_\ell(V_\ell\cap\om))}^2\big)^{1/2}<\infty\qquad,$$
where $E_I^\ell$ denote the component functions
of $(h_\ell^\me)^*E=E_I^\ell\pd x^I$ (sum convention) with respect to Cartesian coordinates.
Here we introduced an obvious (ordered) multi index notation $\d x^I=\d x^{i_1}\wedge\dots\wedge\d x^{i_q}$
for $I:=(i_1,\dots,i_q)\in\{1,\dots,N\}^q$\,.
Transformation theorems and \cite[Satz 4.1]{wloka} for scalar functions
show that this definition is independent of the chosen charts.
Another covering yields the same Sobolev space but with an equivalent norm.
Furthermore, for all $m\in\nzn$ and any $\pc{m+1}$-diffeomorphism $\tau:\tilde{\om}\to\xo$ there exists a constant $c>0$\,, such that
\beq c^\me\normu{E}{\hqm\vono}\le\normu{\tau^*E}{\hqm(\tilde{\om})}\le c\normu{E}{\hqm\vono}\mylabel{pullnorm}\eeq
holds for all $E\in\hqm\vono$\,.

\begin{defini}\mylabel{boundarydefi}
Let $m\in\nzn$\,. We call $\dom$ a `$\pc{m}$-{\sf boundary}', if $\dom$ is a $(N-1)$-dimensional $\pc{m}$-submanifold of $M$\,,
i.e. for each $x\in\dom$ there exists a $\pc{m}$-boundary chart
$(V,h)$ with $h(x)=0$ and $h(\xolv)=\ol{U}_1$\,, such that
$$h(\dom\cap V)=U_1^0\qtext{,}h(\om\cap V)=U_1^-\qtext{,}h\big((M\ohne\ol{\om})\cap V\big)=U_1^+$$
and $h\circ k^\me\in\pc{m}\big(k(\tilde{V}\cap V),\rN\big)$ hold for all charts (for $\om$) $(\tilde{V},k)$ of $x\in\dom$\,.
\end{defini}

Here $U_r\subset\rN$ denotes the open ball centered at the origin with radius $r>0$ and we define
$$U_r^\pm:=\setb{x\in U_r}{\pm x_N>0}\qqtext{,}U_r^0:=\setb{x\in U_r}{x_N=0}\qquad.$$

Using sufficiently smooth restricted boundary charts and following the ideas of the definition of $\hqm(\om)$
we may also introduce for all $m\in[0,\infty)$ the Sobolev spaces $\hqm(\dom)$\,.

We also define $\qh{-m}{q}{}{}(\dom)$ for $m\in(0,\infty)$ as the dual space of
$\qh{m}{q}{}{\circ}(\dom)=\qh{m}{q}{}{}(\dom)$ and introduce the exterior derivative, co-derivative
and star-operator on $\qh{-m}{q}{}{}(\dom)$ by weak formulations.
Utilizing boundary charts, \eqref{pullnorm} and the corresponding results for scalar Sobolev spaces,
e.g. \cite[Satz 8.7, Satz 8.8]{wloka}, which will be applied componentwise to $q$-forms in $\rN$\,,
we obtain the following lemma:

\begin{lem}\mylabel{hspur}
Let $m\in\nz$ and $\Omega$ possess  a $\pc{m+1}$-boundary. Moreover, let
$\iota:\dom\hookrightarrow\ol{\om}$ denote the natural embedding.
Then there exists a linear and continuous tangential trace operator
$$\xgt:\hqm\vono\to\qh{m-1/2}{q}{}{}\vonpo$$
satisfying $\xgt\Phi=\iota^*\Phi$ and $\rot_{\,\,\dom}\xgt\Phi=\xgt\rot\Phi$ for all $\Phi\in\xciq\vonolo$\,,
the vector space of all $\xciq(M)$-forms restricted to $\om$\,.
Moreover, $\xgt$ is surjective, i.e. there exists a linear and continuous tangential extension operator
$$\xchgt:\qh{m-1/2}{q}{}{}\vonpo\to\hqm\vono$$
with the property $\xgt\xchgt=\id$ (right inverse).
\end{lem}

By the star operator we define linear and continuous normal trace and extension operators by
$$\xgn:=(-1)^{(q-1)N}*_{\dom}{\xgt}*:\hqm\vono\To\qh{m-1/2}{q-1}{}{}\vonpo\qquad,$$
$$\xchgn:=(-1)^{q(N-q)}*\xchgt*_{\dom}:\qh{m-1/2}{q-1}{}{}\vonpo\To\hqm\vono\qquad,$$
which possess the corresponding properties. By Stokes' theorem we obtain
\beq\spu{\xrot E}{H}{\xlqp\vono}+\spu{E}{\xdiv H}{\xlq\vono}
=\spu{\xgt E}{\xgn H}{\xlq\vonpo}\mylabel{TN}\eeq
for $\EH\in\qh{1}{q,q+1}{}{}\vono$\,, if $\om$ has a $\pc{2}$-boundary.

It is well known that this suggests to define the tangential trace
$$\xgt E\in\qh{-1/2}{q}{}{}\vonpo$$
of a $q$-form $E\in\xrq\vono$ by
\beq\xgt E(\vp)=\spu{\xgt E}{\vp}{\qh{-1/2}{q}{}{}\vonpo}:=\spu{\xrot E}{\xchgn\vp}{\xlqp\vono}
+\spu{E}{\xdiv\xchgn\vp}{\xlq\vono}\mylabel{tantracedef}\eeq
for all $\vp\in\qh{1/2}{q}{}{}\vonpo$\,. Clearly acting on $E\in\qh{1}{q}{}{}\vono$ it satisfies
\beq\spu{\xgt E}{\vp}{\qh{-1/2}{q}{}{}\vonpo}=\spu{{\xgt} E}{\vp}{\xlq\vonpo}\mylabel{traceonhone}\eeq
for all $\vp\in\qh{1/2}{q}{}{}\vonpo$\,. Hence in this case we have
$\xgt E=\spu{{\xgt} E}{\,\cdot\,}{\xlq\vonpo}$ and we identify the continuous linear functional $\xgt E$
with the element $\xgt E\in\qh{1/2}{q}{}{}\vonpo$\,. We note that $\xgt$ still commutes with the
exterior derivative and that the mapping
$$\xgt:\xrq\vono\To\xxdq\vonpo:=\setb{e\in\qh{-1/2}{q}{}{}\vonpo}{\Rot e\in\qh{-1/2}{q+1}{}{}\vonpo}$$
is continuous. Moreover, we have for all $E\in\xrq\vono$
\beq\xgt E=0\qquad\Equi\qquad E\in\xrnq\vono\qquad,\mylabel{spurnullr}\eeq
where we set
$$\xrnq\vono:=\ol{\cqun(\om)}\qquad,$$
taking the closure in $\rqom{}$\,. We note
$$\xrnq\vono=\setb{E\in\rqom{}}{\forall\;H\in\dqpeom{}\quad\skp{\rot E}{H}_{\lzqpeom}+\skp{E}{\pdiv H}_{\lzqom}=0}$$
and define $\ronqnom{}:=\ronqom{}\cap\rqnom{}$\,.

\begin{defini}\mylabel{transdefi}
Let $m\in\nzn$\,. We call a transformation $\eps$ {\sf admissible}, if and only if
\begin{itemize}
\item $\eps(x)$ is a linear mapping on $q$-forms for all $x\in\om$\,,
\item $\eps$ possesses real $\text{\rm L}^\infty(\Omega)$-coefficients, i.e. the matrix representation of
$\eps$ corresponding to an arbitrary chart basis $\{\pd h^I\}$ has $\text{\rm L}^\infty(\Omega,\rz)$-entries,
\item $\eps$ is symmetric, i.e. for all $E,H\in\lzqom$ we have
$$\skp{\eps E}{H}_{\lzqom}=\skp{E}{\eps H}_{\lzqom}\qquad,$$
\item $\eps$ is uniformly positive definite, i.e.
$$\exists\;c>0\quad\forall\;E\in\lzqom\qquad\skp{\eps E}{E}_{\lzqom}\geq c\norm{E}_{\lzqom}^2\qquad.$$
\end{itemize}
We call $\eps$ $\pc{m}$-{\sf admissible},
if and only if $\eps$ is admissible and has $\com{m}$-coefficients,
which are bounded together with all their derivatives up to the boundary.
Here we mean componentwise differentiation and write $\pa\eps$ for $|\alpha|\leq m$\,.
\end{defini}

We note that admissible transformations $\eps$ generate an equivalent scalar product on $\lzqom$ by
$$\EH\longmapsto\skp{\eps E}{H}_{\lzqom}\qquad.$$

Of course most of these concepts extend to manifolds, whose closures are not compact.
Particularly we may consider the special case of $M:=\rN$ as a smooth Riemannian manifold
of dimension $N\in\nz$ and an exterior domain $\Omega\subset\rN$\,, i.e. $\om$ is connected and $\rN\ohne\Omega$ compact.
The definitions of spaces carry over to exterior domains as long as the compactness of $\omq$
is not necessary.

Using the weight function
$$\rho:=(1+r^2)^{1/2}\qqtext{,}r(x):=|x|$$
we introduce for $m\in\nzn$ and $s\in\rz$ the scalar weighted Sobolev spaces
\begin{align*}
\hmsom&:=\setb{u\in\Lzlocom}{\rho^{s+|\alpha|}\p^\alpha u\in\lzom\text{ for all }|\alpha|\leq m}\qquad,\\
\subset\Hmsom&:=\setb{u\in\Lzlocom}{\rho^s\p^\alpha u\in\lzom\text{ for all }|\alpha|\leq m}
\end{align*}
utilizing the usual multi index notation for partial derivatives. 
(To distinguish between these different polynomially weighted Sobolev spaces of exterior domains
we will use roman and bold roman letters simultaneously.)
Equipped with their natural scalar products these are Hilbert spaces.

Now we have a global chart $(\Omega,\id)$ and $\Omega$ becomes naturally a $N$-dimensional smooth Riemannian manifold with Cartesian
coordinates $\{x_1,\dots,x_N\}$\,. As before with componentwise partial derivatives $\p^\alpha u=(\p^\alpha u_I)\pd x^I$\,, if $u=u_I\pd x^I$\,,
we introduce for $m\in\nzn$ and $s\in\rz$ componentwise the Sobolev spaces $\hqmsom$ resp. $\Hqmsom$ of $q$-forms.
In the special case $m=0$ we define
$$\Lzqsom:=\qhom{0}{q}{s}{}=\qHom{0}{q}{s}{}\qquad.$$
Then for $f=f_I\pd x^I,g=g_I\pd x^I\in\Lzqsom$ we have the scalar product
$$\skp{f}{g}_{\Lzqsom}=\intom\rho^{2s}\ub{f\wedge*\ol{g}}_{=:*\skp{f}{g}_q}=\intom\rho^{2s}\skp{f}{g}_q\,d\lambda
=\intom\rho^{2s}f_I\ol{g}_I\,d\lambda\quad,$$
where $\lambda$ denotes Lebesgue's measure in $\rN$\,.

Furthermore, for $s\in\rz$ we need some special weighted Sobolev spaces suited for the exterior derivative and co-derivative:
\begin{align*}
\rqsom&:=\setb{E\in\Lzqsom}{\rot E\in\Lzqpeom{s+1}}\\
\subset\Rqsom&:=\setb{E\in\Lzqsom}{\rot E\in\Lzqpesom}\\
\dqsom&:=\setb{H\in\Lzqsom}{\pdiv H\in\qLzom{q-1}{s+1}}\\
\subset\Dqsom&:=\setb{H\in\Lzqsom}{\pdiv H\in\qLzom{q-1}{s}}
\end{align*}
Equipped with their natural graph norms these are all Hilbert spaces. To generalize the homogeneous tangential
boundary condition we introduce again $\ronqsom$ resp. $\Ronqsom$ as the closure of $\cqun(\Omega)$ with respect to
the corresponding graph norm $\norm{\,\cdot\,}_{\rqsom}$\,, $\norm{\,\cdot\,}_{\Rqsom}$\,, respectively.
The spaces $\Rqsom$\,, $\Dqsom$ and even $\Ronqsom$ are invariant under multiplication
with bounded smooth functions. As in the last section a subscript $0$ at the lower left corner
indicates vanishing exterior derivative resp. co-derivative, e.g.
$$\ronqsnom=\setb{E\in\ronqsom}{\rot E=0}=\Ronqsnom\qquad.$$

The properties `admissible' and `$\pc{m}$-admissible' extend analogously to our exterior domain case as well.
Nevertheless we need some additional decay properties of our transformations.

\begin{defini}\mylabel{transdefiweig}
Let $m\in\nzn$ and $\tau\geq0$\,. We call $\eps$ $\tau$-$\pc{m}$-{\sf admissible of first} resp. {\sf second kind}, if and only if
$\eps=\eps_{0}+\epsd$ with some $\eps_{0}>0$ is $\pc{m}$-admissible and the perturbation $\epsd$ satisfies
$$\forall\,|\alpha|\leq m\qquad\pa\epsd=\calO(r^{-\tau})\qtext{resp.}\calO(r^{-(\tau+|\alpha|)})\qqtext{as}r\to\infty\qquad.$$
In each case we call $\tau$ the {\sf order of decay} of the perturbation $\epsd$\,.
Without loss of generality we may assume $\eps_{0}=1$\,, i.e. $\eps=\id+\epsd$\,, throughout this paper.
\end{defini}

We note that a transformation is $0$-$\pc{m}$-admissible of first kind, if and only if it is $\pc{m}$-admissible.

Finally if the exterior domain $\om$ has got a $\pc{2}$-boundary
there exist adequate trace and extension operators as well. By obvious restriction, extension by zero
and cutting techniques we obtain linear and continuous tangential trace and extension operators
$$\bigcup_{s\in\rz}\Hqmsom\overset{\xgt}{\To}\qh{m-1/2}{q}{}{}(\dom)\overset{\xchgt}{\To}\bigcap_{s\in\rz}\hqmsom\qqtext{,}\xgt\xchgt=\id\qquad,$$
where $\xchgt$ even maps to compactly supported forms and $\xgt$ even operates on $\hqmlocom$\,.
Here continuity is to be understood in the sense of
\beq\Hqmsom\overset{\xgt}{\To}\qh{m-1/2}{q}{}{}(\dom)\overset{\xchgt}{\To}\hqmsom\mylabel{gthdefaussen}\eeq
for all $s\in\rz$\,. Again by the star operator we get the corresponding
linear and continuous normal trace and extension operators $\xgn:=\pm*_{\dom}{\xgt}*$\,,
$\xchgn:=\pm*\xchgt*_{\dom}$\,. As indicated above by Stokes' theorem \eqref{TN}
we then get for all $s\in\rz$ a linear and continuous tangential trace operator $\xgt:\Rqsom\To\qh{-1/2}{q}{}{}\vonpo$\,,
which is (well) defined by
$$\xgt E(\vp)=\spu{\xgt E}{\vp}{\qh{-1/2}{q}{}{}\vonpo}:=\spu{\xrot E}{\xchgn\vp}{\xlqp\vono}
+\spu{E}{\xdiv\xchgn\vp}{\xlq\vono}$$
for all $E\in\Rqsom$ and $\vp\in\qh{1/2}{q}{}{}\vonpo$\,.
Once more for $E\in\Hqmsom$ we identify the continuous linear functional $\xgt E$
with the element $\xgt E\in\qh{1/2}{q}{}{}\vonpo$ and of course the mapping
$$\xgt:\Rqsom\To\xxdq\vonpo$$
is continuous as well. We still have for all $s\in\rz$ and all $E\in\Rqsom$
\beq\xgt E=0\qquad\Equi\qquad E\in\Ronqsom\qquad.\mylabel{spurnullraussen}\eeq

\section{Regularity theorems}\mylabel{results}

\begin{theo}\mylabel{regularitybd}
Let $m\in\nzn$\,, $\om$ be a connected open subset with compact closure and $\pc{m+2}$-boundary
of some smooth Riemannian manifold $M$ as well as $\eps$ be some $\pc{m+1}$-admissible transformation. 
Furthermore, let
$$E\in\rqom{}\cap\eps^\me\dqom{}$$
with
$$\rot E\in\qhom{m}{q+1}{}{}\qtext{,}\pdiv\eps E\in\qhom{m}{q-1}{}{}\qtext{,}\xgt E\in\qh{m+1/2}{q}{}{}(\dom)\quad.$$
Then $E\in\qhom{m+1}{q}{}{}$ and there exists a positive constant $c$ independent of $E$\,, such that
\begin{align*}
&\qquad\norm{E}_{\qhom{m+1}{q}{}{}}\\
&\leq c\big(\norm{E}_{\Lzqom{}}+\norm{\rot E}_{\qhom{m}{q+1}{}{}}+\norm{\pdiv\eps E}_{\qhom{m}{q-1}{}{}}
+\norm{\xgt E}_{\qh{m+1/2}{q}{}{}(\dom)}\big)\quad.
\end{align*}
\end{theo}

\begin{theo}\mylabel{regularityunbd}
Let $s\in\rz$\,, $m\in\nzn$\,, $\Omega\subset\rN$ be an exterior domain with $\pc{m+2}$-boundary
and $\eps$ be some $\pc{m+1}$-admissible transformation. Furthermore, let
$$E\in\Rqsom\cap\eps^\me\Dqsom\qqtext{with}\xgt E\in\qh{m+1/2}{q}{}{}(\dom)\qquad.$$
\begin{itemize}
\item[\rm\bf (i)] Then $\rot E\in\qHom{m}{q+1}{s}{}$ and $\pdiv\eps E\in\qHom{m}{q-1}{s}{}$ imply $E\in\qHom{m+1}{q}{s}{}$ and with some constant $c>0$
\begin{align*}
&\qquad\norm{E}_{\qHom{m+1}{q}{s}{}}\\
&\leq c\big(\norm{E}_{\Lzqom{s}}+\norm{\rot E}_{\qHom{m}{q+1}{s}{}}
+\norm{\pdiv\eps E}_{\qHom{m}{q-1}{s}{}}+\norm{\xgt E}_{\qh{m+1/2}{q}{}{}(\dom)}\big)
\end{align*}
holds uniformly with respect to $E$\,.
\item[\rm\bf (ii)] If additionally $\eps$ is $0$-$\pc{m+1}$-admissible of second kind and $\tau$-$\pc{0}$-admissible
of first (or second) kind with some $\tau>0$
then $\rot E\in\qhom{m}{q+1}{s+1}{}$ and $\pdiv\eps E\in\qhom{m}{q-1}{s+1}{}$ imply $E\in\qhom{m+1}{q}{s}{}$
and there exists some positive constant $c$\,, such that the estimate
\begin{align*}
&\qquad\norm{E}_{\qhom{m+1}{q}{s}{}}\\
&\leq c\big(\norm{E}_{\Lzqom{s}}+\norm{\rot E}_{\qhom{m}{q+1}{s+1}{}}
+\norm{\pdiv\eps E}_{\qhom{m}{q-1}{s+1}{}}+\norm{\xgt E}_{\qh{m+1/2}{q}{}{}(\dom)}\big)
\end{align*}
holds uniformly with respect to $E$\,.
\end{itemize}
\end{theo}

\begin{rem}\mylabel{regularityrem}
Utilizing the transformation $E\leadsto\eps E$ and/or the Hodge star-operator we obtain similar
results for spaces like $\eps^\me\rqom{}\cap\dqom{}$ and/or with prescribed normal traces $\xgn$\,.
\end{rem}

\section{Proofs}

\subsection{Riemannian manifolds with compact closure}

\begin{proof}{\bf of Theorem \ref{regularitybd} } 
Extending the boundary form $\xgt E$ to $\om$ by Lemma \ref{hspur} via
$$\Eh:=\xchgt\xgt E\in\qhom{m+1}{q}{}{}$$
yields that $\Es:=E-\Eh$ is an element of $\ronqom{}\cap\eps^\me\dqom{}$ and still satisfies
$$\rot\Es\in\qhom{m}{q+1}{}{}\qqtext{,}\pdiv\eps\Es\in\qhom{m}{q-1}{}{}\qquad.$$
Hence our problem is reduced to the discussion of forms with homogeneous tangential trace.

The classical case $N=3$\,, $q=1$ and $\om$ is some bounded domain in $\rd$ has been proved by Weber
in \cite{weberreg} using the natural regularity of $(q-1=0)$- resp. $(q+2=3)$-forms,
i.e. scalar functions. Here in the generalized case we have to deal with some additional difficulties.

Using a partition of unity we localize our problem and only consider the more difficult case of boundary charts.
(A very simple proof of inner regularity utilizing Fourier's transformation is presented in section \ref{extdomsec}.)
By \eqref{pullnorm} and Lemma \ref{taudivrot} we transform our problem to the special domain $U_1^-\subset U_1\subset\rN$
using a $\xsc{m+2}{}$-boundary chart. Hence we have to show the following assertion for the model problem:

\begin{lem}\mylabel{hmqziel}
Let $\eps$ be $\pc{m+1}$-admissible (in $U_1^-$) and $E\in\xrnq(U_1^-)\cap\eps^{-1}\xdq(U_1^-)$
with $\supp E\subset \ol{U_\varrho^-}$ for some $\varrho\in(0,1)$  as well as
$$\xrot E\in\xhmqp(U_1^-)\qqtext{,}\xdiv\eps E\in\xhmqm(U_1^-)\qquad.$$
Then $E\in\xhmpq(U_1^-)$ and there exists a positive constant $c$\,, such that
$$\normu{E}{\xhmpq(U_1^-)}\le c\big(\normu{E}{\xlq(U_1^-)}+\normu{\xrot E}{\xhmqp(U_1^-)}
+\normu{\xdiv\eps E}{\xhmqm(U_1^-)}\big)$$
holds uniformly with respect to $E$\,.
\end{lem}

\begin{proof}
First let us discuss the case $N\ge 3$ by induction over $q$ and $m$\,.
Since we have $\xr{0}{}{\circ}(U_1^-)=\hon{1}{}(U_1^-)$ ($\rot$ acts as $\nabla$\,!) the case $q=0$ is trivial.
Moreover, because of $\xd{N}{}{}(U_1^-)=\h{1}{}{}(U_1^-)$ ($\pdiv$ acts as $\nabla$\,!) the case $q=N$ is trivial as well.
Thus we may assume $1\leq q\leq N-1$ and that the assertion is valid for $q-1$\,.
Let $m=0$\,. First we take care about the tangential derivatives and show
\begin{align}\begin{split}
\p_iE&\in\xlq(U_1^-)\qquad,\\
\norm{\p_iE}_{\xlq(U_1^-)}&\le c\normu{E}{\xrq(U_1^-)\cap\eps^{-1}\xdq(U_1^-)}
\end{split}\mylabel{tanabl}\end{align}
for $i=1,\dots,N-1$\,. By symmetry it is sufficient to consider $i=1$\,.
We choose some $\theta\in(0,1)$ satisfying $\varrho+4\theta<1$
and put $\varrho_j:=\varrho+j\theta$\,, $j=1,\dots,4$\,. For $0<|h|<\theta$ we introduce the mappings
$$\Abb{\tau_h}{\rz^N_-}{\rz^N_-}{x}{(x_1+h,x_2,\cdots,x_N)}\qqtext{,}
\delta_h:=\frac{1}{h}(\tau_h-\id)\qquad,$$
where $\rz^N_-:=\set{x\in\rN}{x_N<0}$\,. The pullback $\delta_h^*$ of the latter operator acts componentwise
as the differential quotient and commutates with $\xrot$\,, $*$ and thus also with $\xdiv$\,.
For all $F,G\in\xlq(U_1^-)$ with support in $\ol{U^-_{\varrho_3}}$ we have
with some constant $c>0$ independent of $h$ or $F$
\begin{align}\begin{split}
\spu{\delta_h^*F}{G}{\xlq(U_1^-)}&=-\spu{F}{\delta_{-h}^*G}{\xlq(U_1^-)}\qquad,\\
\delta_h^*\eps F&=\eps\delta_h^*F+(\delta_h\eps)\tau_h^*F\qquad,\\
\norm{\tau_h^*F}_{\xlq(U_1^-)}&\le c\norm{F}_{\xlq(U_1^-)}\qquad,\\
\normb{(\delta_h\eps)F}_{\xlq(U_1^-)}&\le c\norm{F}_{\xlq(U_1^-)}\qquad,
\end{split}\mylabel{deltaepsab}\end{align}
where $(\delta_h\eps)\Phi(x):=\big(\delta_h{\eps}_{J,I}(x)\big)\Phi_I(x)\d x^J$
with $\Phi(x)=\Phi_I(x)\dx^I$ and the matrix entries ${\eps}_{I,J}$ of $\eps$\,.
Following in straight lines \cite[Theorem 3.13]{agmon} we obtain for $m\in\nz$ and all
$F\in\xhmq(U_1^-)$ supported in $\ol{U^-_{\varrho_3}}$
$$\normu{\delta_h^*F}{\xhmmq(U_1^-)}\le\normu{F}{\xhmq(U_1^-)}\qquad.$$
To show \eqref{tanabl} by \cite[Theorem 3.15]{agmon} it suffices to prove
$$\norm{\delta_h^*E}_{\xlq(U^-_{\varrho_1})}\leq c\normu{E}{\xrq(U_1^-)\cap\eps^{-1}\xdq(U_1^-)}\qquad,$$
where $c>0$ is independent of $h$\,, $\varrho$ or $E$\,.
In turn this estimate follows by the even stronger estimate
\beq\big|\spu{\eps\delta_h^*E}{\Phi}{\xlq(U^-_{\varrho_1})}\big|
\leq c\normu{E}{\xrq(U_1^-)\cap\eps^{-1}\xdq(U_1^-)}\normu{\Phi}{\xlq(U^-_{\varrho_1})}\mylabel{zielabsch}\eeq
for all $\Phi\in\lzq(U^-_{\varrho_1})$\,, where $c>0$ is independent of $h$\,, $\varrho$\,, $E$ or $\Phi$\,.
Therefore, let $\Phi\in\lzq(U^-_{\varrho_1})$\,. According to Lemma \ref{lzzerlegungen} we decompose $\Phi$ 
(actually the extension by zero to $U_1^-$ of $\Phi$) orthogonally in $\lzq(U_1^-)$
$$\Phi=\Phi_1+\eps^\me\Phi_2\qquad,$$ 
where $\Phi_1\in\ol{\rot\xrnqm(U_1^-)}$ and $\Phi_2\in\ol{\pdiv\xdqp(U_1^-)}$
\big(closures in $\lzq(U_1^-)$\big), since $\dH{q}{}{}(U_1^-)$ vanishes by \cite[Satz 1, Satz 2]{picardharmdiff}
and thus $\dH{q}{}{\eps}(U_1^-)=\{0\}$ as well.
Moreover, by \eqref{rotabg}, \eqref{divabg}
we may assume $\Phi_1=\rot\Psi_1$ and $\Phi_2=\pdiv\Psi_2$ with
$\Psi_1\in\xrnqm(U_1^-)\cap\xdqmn(U_1^-)$ and $\Psi_2\in\xdqp(U_1^-)\cap\xrnqpn(U_1^-)$\,. Furthermore,
\eqref{abschaetzrotunddiv} yields a constant $c>0$ independent of $\Phi$\,, $\Phi_\ell$\,, $\Psi_\ell$\,, such that
$$\normu{\Psi_1}{\xrqm(U_1^-)}+\normu{\Psi_2}{\xdqp(U_1^-)}\le c\normu{\Phi}{\xlq(U^-_{\varrho_1})}$$
holds. Let $\chi\in\xsc{\infty}{\circ}(U_{\varrho_2})$ with $\restr{\chi}{U^-_{\varrho_1}}=1$\,.
Then the assumption of the induction for $\eps=\id$ yields $\Psi_1,\chi\Psi_1\in\qh{1}{q-1}{}{}(U_1^-)$ and
$$\normu{\chi\Psi_1}{\qh{1}{q-1}{}{}(U_1^-)}\leq c\normu{\Psi_1}{\qh{1}{q-1}{}{}(U_1^-)}
\le c\normu{\Psi_1}{\xrqm(U_1^-)}\le c\normu{\Phi}{\xlq(U^-_{\varrho_1})}\qquad.$$
Clearly the form $\chi\Psi_2$ possesses compact support in $U_{\varrho_2}^-\cup U_{\varrho_2}^0$
and by Lemma \ref{spiegel} and \eqref{divspiegel} the extension by zero of $S_{\pdiv}\chi\Psi_2$ to $\rN$
is an element of $\Dqpe{}(\rN)$\,. Hence we have $\tilde{\Phi}_2:=\xdiv S_{\pdiv}\chi\Psi_2\in\xdqn(\rN)$
with $\supp\tilde{\Phi}_2\subset U_{\varrho_2}$ and $\tilde{\Phi}_2\big|_{U^-_{\varrho_1}}=\Phi_2$\,.
Lemma \ref{fourier} yields some $(q+1)$-form $H\in\xh{1}{q+1}{}(\rN)$ satisfying $\xdiv H=\tilde{\Phi}_2$ and
furthermore the estimate $\normu{H}{\xh{1}{q+1}{}(\rN)}\le c\normu{\Phi}{\xlq(U^-_{\varrho_1})}$\,.
Using $\Phi=\rot\chi\Psi_1+\eps^\me\pdiv\chi H$ in $U^-_{\varrho_1}$ and \eqref{deltaepsab} as well as
$\delta_{-h}^*(\chi\Psi_1)\in\xrnqm(U_1^-)$\,, $E\in\xrnq(U_1^-)$ we get
\begin{align*}
&\qquad\qquad\skp{\eps\delta_h^*E}{\Phi}_{\xlq(U^-_{\varrho_1})}\\
&=\skpb{\delta_h^*(\eps E)}{\Phi}_{\xlq(U^-_{\varrho_1})}
-\skpb{(\delta_h\eps)\tau_h^*E}{\Phi}_{\xlq(U^-_{\varrho_1})}\\
&=-\skpb{\eps E}{\xrot\delta_{-h}^*(\chi\Psi_1)}_{\xlq(U_1^-)}
-\skpb{E}{\xdiv\delta_{-h}^*(\chi H)}_{\xlq(U_1^-)}\\
&\qquad-\skpb{\eps E}{(\delta_{-h}\eps^{-1})\tau_{-h}^*\xdiv\chi H}_{\xlq(U_1^-)}
-\skpb{(\delta_h\eps)\tau_h^*E}{\Phi}_{\xlq(U^-_{\varrho_1})}\\
&=\skpb{\pdiv\eps E}{\delta_{-h}^*(\chi\Psi_1)}_{\xlq(U_1^-)}
+\skpb{\rot E}{\delta_{-h}^*(\chi H)}_{\xlq(U_1^-)}\\
&\qquad-\skpb{\eps E}{(\delta_{-h}\eps^{-1})\tau_{-h}^*\xdiv\chi H}_{\xlq(U_1^-)}
-\skpb{(\delta_h\eps)\tau_h^*E}{\Phi}_{\xlq(U^-_{\varrho_1})}\qquad,
\end{align*}
which immediately implies \eqref{zielabsch}. Hence \eqref{tanabl} is proved.

The normal partial derivative $\p_NE$ may be discussed as follows. By the usual formula
$$\d E=\d(E_I\,\d x^I)=\p_jE_I\,\d x^j\wedge\d x^I=(\pm\p_jE_I)\,\d x^{I+j}$$
we get
\beq\pm\p_NE_I=(\rot E)_{I+N}-\sum_{I\ni j=1}^{N-1}\pm\p_jE_{I+N-j}\in\lz(U_1^-)\mylabel{normalablE}\eeq
for all $I\not\ni N$ and thus $E^\tau\in\qh{1}{q}{}{}(U_1^-)$ with the decomposition from \eqref{taurhodecodef}.
The usual formula for the co-derivative reads
$$\delta H=\delta(H_I\,\d x^I)=(\pm\p_jH_I)*(\d x^j\wedge*\d x^I)=(\pm\p_jH_I)\,\d x^{I-j}\qquad.$$
By $\p_i(\eps E)=(\p_i\eps)E+\eps\p_iE$ we obtain $\p_i(\eps E)\in\lzq(U_1^-)$
for $i=1,\dots,N-1$ and hence
\beq\pm\p_N(\eps E)_I=(\pdiv\eps E)_{I-N}-\sum_{I\not\ni j=1}^{N-1}\pm\p_j(\eps E)_{I-N+j}\in\lz(U_1^-)\mylabel{normalablepsE}\eeq
for all $I\ni N$\,. Therefore, $(\eps E)^\rho\in\qh{1}{q}{}{}(U_1^-)$\,.

Now Lemma \ref{hmschluss} yields $E\in\qh{1}{q}{}{}(U_1^-)$ and the case $m=0$ is proved.

Let $m\geq1$ and our assertions be valid for $m-1$ as well as the assumptions be given for $m$\,.
We consider $E,\eps E\in\xhmq(U_1^-)$ with $E\in\ronq{}(U_1^-)\cap\eps^\me\pdq{}(U_1^-)$\,,
$\supp E\subset \ol{U_\varrho^-}$ and
$$\xrot E\in\xhmqp(U_1^-)\qqtext{,}\xdiv\eps E\in\xhmqm(U_1^-)\qquad.$$
Moreover, we have the estimate
$$\normu{E}{\xhmq(U_1^-)}\le c\big(\normu{E}{\xlq(U_1^-)}+\normu{\xrot E}{\xhmmqp(U_1^-)}+\normu{\xdiv\eps E}{\xhmmqm(U_1^-)}\big)\qquad.$$
For sufficiently small $h$ we have $\delta_h^*E\in\ronq{}(U_1^-)$ and $\delta_h^*E$ resp. $\delta_h^*\rot E$
converges weakly to $\p_1E$ resp. $\p_1\rot E$ in $\lzq(U_1^-)$ resp. $\lzqpe(U_1^-)$ as $h\to0$\,.
Thus $\p_1E\in\ronq{}(U_1^-)$ and $\rot\p_1E=\p_1\rot E$\,. 
Analogously we get $\p_iE\in\ronq{}(U_1^-)$ and $\rot\p_iE=\p_i\rot E$ 
for all indices $i=2,\dots,N-1$\,.
Hence all tangential derivatives $\p_iE\in\ronq{}(U_1^-)\cap\eps^\me\pdq{}(U_1^-)$\,, $i=1,\dots,N-1$\,, satisfy
\begin{align*}
\rot\p_iE=\p_i\rot E\in\xhmmqp(U_1^-)\qquad,\\
\xdiv\eps\p_iE=\p_i\xdiv\eps E-\xdiv(\p_i\eps)E\in\xhmmqm(U_1^-)\qquad,
\end{align*}
which implies $\p_iE\in\xhmq(U_1^-)$ and also $\p_i(\eps E)\in\xhmq(U_1^-)$ by assumption.
By \eqref{normalablE} and \eqref{normalablepsE} we obtain $\p_NE^\tau,\p_N(\eps E)^\rho\in\xhmq(U_1^-)$
and thus $E^\tau,(\eps E)^\rho\in\xhmpq(U_1^-)$ as well.
Finally we achieve by Lemma \ref{hmschluss} $E\in\xhmpq(U_1^-)$\,,
which completes the induction and hence the proof for $N\geq3$\,.

The only non trivial remaining case is $N=2$\,, $q=1$\,. But this case can be proved
similarly to the case $N\ge3$ without using Lemma \ref{fourier},
since then $\Psi_2$ is even an element of $\pdi{2}{}{}{}(U_1^-)=\qh{1}{2}{}{}(U_1^-)$\,.
\end{proof}

\end{proof}

\subsection{Exterior domains}

\begin{proof}{\bf of Theorem \ref{regularityunbd} }\mylabel{extdomsec}
Extending the boundary form $\xgt E$ to $\om$ by \eqref{gthdefaussen} via
$$\Eh:=\xchgt\xgt E\in\qhom{m+1}{q}{s}{}\qqtext{,}\supp\Eh\text{ compact}$$
yields that $\Es:=E-\Eh$ satisfies the assumptions of Theorem \ref{regularityunbd} with homogeneous tangential trace.
Hence we may assume $\xgt E=0$\,, i.e. $E\in\Ronqsom\cap\eps^\me\Dqsom$\,, since $\xchgt$ is continuous.

Let us assume for a moment that Theorem \ref{regularityunbd} holds in the special case $\om=\rN$\,.
Moreover, let $\eta$ denote a smooth cut-off function, which vanishes near $\dom$ and equals $1$ near infinity.
Then by Theorem \ref{regularityunbd} in the whole space case
$\eta E\in\qH{m+1}{q}{s}{}(\rN)$ resp. $\eta E\in\qh{m+1}{q}{s}{}(\rN)$\,.
Furthermore, Theorem \ref{regularitybd} may be applied to the truncated form
$(1-\eta)E\in\ronq{}(\omb)\cap\eps^\me\pdq{}(\omb)$ with some adequate bounded subdomain $\omb\subset\om$
yielding $(1-\eta)E\in\qh{m+1}{q}{}{}(\omb)$\,.
Then extending $(1-\eta)E$ by zero into the whole of $\om$ leads to $(1-\eta)E\in\qhom{m+1}{q}{s}{}$\,.
The estimates follow by induction.
Hence, our proof is reduced to the following assertion for the special model case $\om=\rN$\,:

\begin{lem}\mylabel{wholespaceregularity}
Let $s\in\rz$\,, $m\in\nzn$ and $\eps$ be some $\pc{m+1}$-admissible transformation as well as
$E\in\Rqs(\rN)\cap\eps^\me\Dqs(\rN)$\,.
\begin{itemize}
\item[\rm\bf (i)] Then $\rot E\in\qH{m}{q+1}{s}{}(\rN)$ and $\pdiv\eps E\in\qH{m}{q-1}{s}{}(\rN)$
imply $E\in\qH{m+1}{q}{s}{}(\rN)$ and with some constant $c>0$
$$\norm{E}_{\qH{m+1}{q}{s}{}(\rN)}\leq c\big(\norm{E}_{\Lzqs(\rN)}+\norm{\rot E}_{\qH{m}{q+1}{s}{}(\rN)}
+\norm{\pdiv\eps E}_{\qH{m}{q-1}{s}{}(\rN)}\big)$$
holds uniformly with respect to $E$\,.
\item[\rm\bf (ii)] If additionally $\eps$ is a $0$-$\pc{m+1}$-admissible transformation of second kind
and $\tau$-$\pc{0}$-admissible of first (or second) kind with some $\tau>0$
then $\rot E\in\qh{m}{q+1}{s+1}{}(\rN)$ and $\pdiv\eps E\in\qh{m}{q-1}{s+1}{}(\rN)$
imply $E\in\qh{m+1}{q}{s}{}(\rN)$ and there exists some positive constant $c$\,, such that the estimate
$$\norm{E}_{\qh{m+1}{q}{s}{}(\rN)}\leq c\big(\norm{E}_{\Lzqs(\rN)}+\norm{\rot E}_{\qh{m}{q+1}{s+1}{}(\rN)}
+\norm{\pdiv\eps E}_{\qh{m}{q-1}{s+1}{}(\rN)}\big)$$
holds uniformly with respect to $E$\,.
\end{itemize}
\end{lem}

\begin{proof}
Our induction over $m$ starts with $m=0$\,.

\begin{lem}\mylabel{regganzraumeps}
Let $\eps$ be $\pc{1}$-admissible. Then $\pRq{}(\rN)\cap\eps^\me\pDq{}(\rN)=\qH{1}{q}{}{}(\rN)$ 
holds with equivalent norms depending on $\eps$\,.
\end{lem}

\begin{proof}
Partial integration, i.e. Stokes' theorem, and the well known formula $\rot\pdiv+\pdiv\rot=\Delta$
(Here the Laplacian $\Delta$ acts componentwise with respect to Euclidian coordinates.) yield
\beq\forall\;\Phi\in\cqun(\rN)\quad\sum_{n=1}^{N}\norm{\p_n\Phi}_{\Lzq{}(\rN)}^2=\norm{\rot\Phi}_{\Lzqpe{}(\rN)}^2
+\norm{\pdiv\Phi}_{\qLz{q-1}{}(\rN)}^2\quad.\mylabel{dnrotdiv}\eeq
A combination of this identity and Fourier's transformation, i.e. \eqref{fouriersieben}-\eqref{fourierneun} 
and \eqref{RTformula}, implies
\beq\pRq{}(\rN)\cap\pDq{}(\rN)=\qH{1}{q}{}{}(\rN)\mylabel{regganzraum}\eeq
with equal norms, since $\cqun(\rN)$ is dense in $\qH{1}{q}{}{}(\rN)$\,.

Now let $E\in\pRq{}(\rN)\cap\eps^\me\pDq{}(\rN)$\,.
By \cite[Lemma 1, Lemma 7]{potential} 
\big(See also \cite{decomposition} as well as Appendix \ref{apphodgedeco} and \ref{appcomp}.\big)
we decompose $E=\rot\Phi+\Psi$ according to
$$\Lzq{}(\rN)=\ol{\rot\pR{q-1}{}{}{}(\rN)}\oplus\dqn{}(\rN)
=\rot\big(\pr{q-1}{-1}{}{}(\rN)\cap\pdi{q-1}{-1}{0}{}(\rN)\big)\oplus\dqn{}(\rN)$$
observing $\rot\Psi=\rot E$ and $\pdiv\Psi=0$\,.
By \eqref{regganzraum} we obtain $\Psi\in\qH{1}{q}{}{}(\rN)$ 
and the estimate $\norm{\Psi}_{\qH{1}{q}{}{}(\rN)}\leq c\norm{E}_{\pRq{}(\rN)}$
with some constant $c>0$\,. Hence $\eps\Psi\in\qH{1}{q}{}{}(\rN)$ and $\Phi$ solves the elliptic system
$$\pdiv\eps\rot\Phi=\pdiv\eps E-\pdiv\eps\Psi=:F\in\qLz{q-1}{}(\rN)\qqtext{,}\pdiv\Phi=0\qquad,$$
where $\norm{F}_{\Lzq{}(\rN)}\leq c\norm{E}_{\pRq{}(\rN)\cap\eps^\me\pDq{}(\rN)}$\,.
Using the operators 
$$\Abb{\tau_{h,i}}{\rN}{\rN}{x}{(x_1,\cdots,x_{i-1},x_{i}+h,x_{i+1},\cdots,x_N)}\qquad,$$
$$\delta_{h,i}:=\frac{1}{h}(\tau_{h,i}-\id)$$
for $i=1,\dots,N$ and $h>0$ defined on $\rN$ 
corresponding to $\tau_{h}=\tau_{h,1}$ and $\delta_{h}=\delta_{h,1}$ defined on $\rz^N_{-}$
from the proof of Theorem \ref{regularitybd}
as well as $\norm{\tau_{h,i}^*\phi}_{\Lzq{}(\rN)}=\norm{\phi}_{\Lzq{}(\rN)}$ and the estimates
$$\norm{\delta_{h,i}^*\phi}_{\Lzq{}(\rN)}\leq\norm{\p_i\phi}_{\Lzq{}(\rN)}\qqtext{,}
\norm{\rot\phi}_{\Lzqpe{}(\rN)}\leq\sum_{n=1}^{N}\norm{\p_n\phi}_{\Lzq{}(\rN)}$$ 
we get
$$\skp{\eps\delta_{h,i}^*\rot\Phi}{\rot\phi}_{\Lzq{}(\rN)}
=\skp{\pdiv\eps\rot\Phi}{\delta_{-h,i}^*\phi}_{\qLz{q-1}{}(\rN)}
-\skpb{\rot\Phi}{(\delta_{-h,i}\eps)\tau_{-h,i}^*\rot\phi}_{\Lzq{}(\rN)}$$
and thus by \eqref{dnrotdiv} uniformly with respect to $\phi$ and $h$

\begin{align*}
\big|\skp{\eps\delta_{h,i}^*\rot\Phi}{\rot\phi}_{\Lzq{}(\rN)}\big|
&\leq c\norm{E}_{\pRq{}(\rN)\cap\eps^\me\pDq{}(\rN)}\sum_{n=1}^{N}\norm{\p_n\phi}_{\qLz{q-1}{}(\rN)}\\
&\leq c\norm{E}_{\pRq{}(\rN)\cap\eps^\me\pDq{}(\rN)}\big(\norm{\rot\phi}_{\Lzq{}(\rN)}+\norm{\pdiv\phi}_{\qLz{q-2}{}(\rN)}\big)
\end{align*}
for all $\phi\in\qc{\infty}{q-1}{\circ}(\rN)$\,.
By this estimate and since $\qc{\infty}{q-1}{\circ}(\rN)$ is a dense subset of 
$\pr{q-1}{-1}{}{}(\rN)\cap\pdi{q-1}{-1}{}{}(\rN)$ we obtain
$$\norm{\delta_{h,i}^*\rot\Phi}_{\Lzq{}(\rN)}\leq c\norm{E}_{\pRq{}(\rN)\cap\eps^\me\pDq{}(\rN)}\qquad,$$
where the constant $c>0$ is independent of $h$\,. Therefore,   $\rot\Phi\in\qH{1}{q}{}{}(\rN)$ and the estimates
$\norm{\p_i\rot\Phi}_{\Lzq{}(\rN)}\leq c\norm{E}_{\pRq{}(\rN)\cap\eps^\me\pDq{}(\rN)}$\,, $i=1,\dots,N$\,, hold,
which completes the proof.
\end{proof}

Now we may proceed with the induction start.
Let $E\in\Rqs(\rN)\cap\eps^\me\Dqs(\rN)$\,. We have $\rho^sE\in\Lzq{}(\rN)$ and by \eqref{RTcomformula}
\begin{align*}
\rot\,(\rho^sE)&=\rho^s\rot\,E+s\rho^{s-2}RE\in\Lzqpe{}(\rN)\qquad,\\
\pdiv(\rho^s\eps E)&=\rho^s\pdiv\eps E+s\rho^{s-2}T\eps E\in\qLz{q-1}{}(\rN)\qquad.
\end{align*}
Thus, using Lemma \ref{regganzraumeps} $\rho^sE\in\pRq{}(\rN)\cap\eps^\me\pDq{}(\rN)=\qH{1}{q}{}{}(\rN)$ follows and
$$\p_n(\rho^sE)=\rho^s\p_nE+s\rho^{s-2}\mathcal{X}_nE\in\Lzq{}(\rN)$$
yields {\bf (i)} with the desired estimates. Looking at 
$$E\in\rqs(\rN)\cap\eps^\me\dqs(\rN)\subset\Rqs(\rN)\cap\eps^\me\Dqs(\rN)$$
we obtain $E\in\qH{1}{q}{s}{}(\rN)$ by {\bf (i)}.
Therefore, it only remains to show $\p_nE\in\Lzq{s+1}(\rN)$ for $n=1,\dots,N$\,.
We choose a real smooth cut-off function $\varphi$ with $\varphi=1$ on $(-\infty,1]$ and $\varphi=0$ on $[2,\infty)$
and set $\eta_{t}:=\varphi(r/t)$\,.
Then we calculate with \eqref{dnrotdiv} or \eqref{regganzraum} uniformly with respect to $t\in\rzp$
\begin{align*}
&\qquad\normb{\p_n(\eta_t E)}_{\Lzq{s+1}(\rN)}\\
&\leq c\Big(\normb{\p_n(\ub{\rho^{s+1}\eta_t E}_{\in\qH{1}{q}{}{}(\rN)})}_{\Lzq{}(\rN)}
+\normb{(s+1)\rho^{s-1}\mathcal{X}_n\eta_t E}_{\Lzq{}(\rN)}\Big)\\
&\leq c\Big(\normb{\rot\,(\rho^{s+1}\eta_t E)}_{\Lzqpe{}(\rN)}
+\normb{\pdiv(\rho^{s+1}\eta_t E)}_{\qLz{q-1}{}(\rN)}+\norm{\eta_t E}_{\Lzq{s}(\rN)}\Big)\\
&\leq c\Big(\norm{\eta_t E}_{\rqs(\rN)\cap\eps^\me\dqs(\rN)}
+\normb{\pdiv(\eta_t\epsd E)}_{\qLz{q-1}{s+1}(\rN)}\Big)\\
&\leq c\Big(\norm{\eta_t E}_{\rqs(\rN)\cap\eps^\me\dqs(\rN)}
+\sum_{m=1}^{N}\normb{\p_m(\eta_t E)}_{\Lzq{s+1-\tau}(\rN)}\Big)\qquad.
\end{align*}
Since $\tau>0$ and decomposing $\rN=\ol{U_\vartheta}\cup A_\vartheta$ we get for all $\vartheta\in\rzp$
\begin{align}
\begin{split}
&\qquad\normb{\p_m(\eta_t E)}_{\Lzq{s+1-\tau}(\rN)}^2\\
&\leq c_\vartheta\normb{\p_m(\eta_t E)}_{\Lzq{s}(\rN)}^2
+(1+\vartheta^2)^{-\tau}\normb{\p_m(\eta_t E)}_{\Lzq{s+1}(\rN)}^2
\end{split}\mylabel{trick}
\end{align}
with some constant $c_\vartheta>0$ depending only on $\vartheta$ and $\tau$\,.
Here we have 
$$A_{\vartheta}:=\rN\ohne\ol{U_{\vartheta}}=\setb{x\in\rN}{|x|>\vartheta}\qquad.$$
A combination of the latter two estimates yields for some sufficient large $\vartheta$ and with {\bf (i)}
\begin{align*}
&\qquad\sum_{n=1}^{N}\normb{\p_n(\eta_t E)}_{\Lzq{s+1}(\rN)}\\
&\leq c\Big(\norm{\eta_t E}_{\qH{1}{q}{s}{}(\rN)}+\normb{\rot\,(\eta_t E)}_{\Lzqpe{s+1}(\rN)}
+\normb{\pdiv(\eta_t\eps E)}_{\qLz{q-1}{s+1}(\rN)}\Big)\\
&\leq c\big(\norm{E}_{\rqs(\rN)\cap\eps^\me\dqs(\rN)}
+\norm{t^\me r^\me RE}_{\Lzqpe{s+1}(Z_{t,2t})}+\norm{t^\me r^\me T\eps E}_{\qLz{q-1}{s+1}(Z_{t,2t})}\big)\quad,
\end{align*}
where $Z_{t,T}:=A_{t}\cap U_{T}=\setb{x\in\rN}{t<|x|<T}$\,.
Using $t^\me\leq2r^\me$ in $Z_{t,2t}$ we finally obtain the estimate
$$\sum_{n=1}^{N}\norm{\p_nE}_{\Lzq{s+1}(U_t)}\leq\sum_{n=1}^{N}\normb{\p_n(\eta_t E)}_{\Lzq{s+1}(\rN)}
\leq c\norm{E}_{\rqs(\rN)\cap\eps^\me\dqs(\rN)}\qquad,$$
which holds uniformly with respect to $t$\,. Thus letting $t\to\infty$ the monotone convergence theorem
implies $E\in\qh{1}{q}{s}{}(\rN)$ and the desired estimate.
Hence {\bf (ii)} is proved and thus the case $m=0$ is completed.
 
For the induction step we assume $\eps$ to be $\pc{m+1}$-admissible and
$$\rot\, E\in\qH{m}{q+1}{s}{}(\rN)\qqtext{,}\pdiv\eps E\in\qH{m}{q-1}{s}{}(\rN)\qquad.$$
The assertion for $m-1$ yields $E\in\qH{m}{q}{s}{}(\rN)$ and the corresponding estimate. Then for $n=1,\dots,N$ we get
$\p_nE\in\Lzqs(\rN)$\,, $\rot\,\p_nE\in\qH{m-1}{q+1}{s}{}(\rN)$ and
$$\pdiv(\eps\p_nE)=\p_n\pdiv\eps E-\pdiv\big((\p_n\eps)E\big)\in\qH{m-1}{q-1}{s}{}(\rN)\qquad.$$
Using once again the assumption for $m-1$ we obtain $\p_nE\in\qH{m}{q}{s}{}(\rN)$ and
\begin{align*}
&\qquad\norm{\p_nE}_{\qH{m}{q}{s}{}(\rN)}\\
&\leq c\Big(\norm{\p_nE}_{\Lzq{s}(\rN)}+\norm{\rot\,\p_nE}_{\qH{m-1}{q+1}{s}{}(\rN)}
+\normb{\pdiv(\eps\p_nE)}_{\qH{m-1}{q-1}{s}{}(\rN)}\Big)
\end{align*}
for $n=1,\dots,N$\,. Hence $E\in\qH{m+1}{q}{s}{}(\rN)$ and

\begin{align*}
\norm{E}_{\qH{m+1}{q}{s}{}(\rN)}&\leq c\big(\norm{E}_{\qH{m}{q}{s}{}(\rN)}+\sum_{n=1}^{N}\norm{\p_nE}_{\qH{m}{q}{s}{}(\rN)}\big)\\
&\leq c\big(\norm{E}_{\qH{m}{q}{s}{}(\rN)}+\norm{\rot\,E}_{\qH{m}{q+1}{s}{}(\rN)}+\norm{\pdiv\eps E}_{\qH{m}{q-1}{s}{}(\rN)}\big)\qquad.
\end{align*}
This shows {\bf (i)}.
Similarly we prove {\bf (ii)} paying attention to the fact
that the weights in the $\norm{\,\cdot\,}_{\qh{m}{q}{s}{}(\rN)}$-norms grow with the number of
derivatives and that this effect is compensated by the decay properties of $\epsd$ and its derivatives.
\end{proof}

\end{proof}

\begin{rem}\mylabel{regganzraumepsrem}
Lemma \ref{regganzraumeps} and Lemma \ref{wholespaceregularity} as well as an obvious cutting technique 
easily yield inner regularity results. These even include weighted inner regularity in exterior domains.
\end{rem}

\appendix

\section{Appendix}

As before let $M$ be a $N$-dimensional smooth Riemannian manifold
and let $\om\subset M$ denote some connected open subset with compact closure in $M$
or some exterior domain of $M=\rN$\,. 
Moreover, throughout this appendix $\nu$ denotes some admissible transformation.

\subsection{Density results}

Let $\om\subset M$ be a connected open subset with compact closure of the $N$-dimensional smooth Riemannian manifold $M$\,.
Using charts and the results and techniques known from the scalar cases, i.e. mollifiers, 
we get the following assertions for $m\in\nzn$\,:
$$\cquom\cap\Hqmom\qqtext{is dense in}\Hqmom\qquad.$$
If $\om$ has the `{\it segment property}', i.e. for each $x\in\dom$ there exist a chart $(V,h)$\,,
some $\varrho\in(0,1)$ and some vector $v\in\rN$ with $h(x)=0$\,, $h(\xolv)=\xolu_1$ and
$$U_\varrho\cap\ol{h(\xo\cap V)}+\tau v\subset h(\xo\cap V)$$
for all $\tau\in(0,1)$ 
\big(Please compare to \cite[Definition 2.1]{agmon} for the classical segment property.
We note that manifolds with $\pc{1}$-boundary possess the segment property.\big),
we can adopt more properties from the scalar cases. For example,
\beq\cqu(\omq)\qqtext{is dense in}\Hqmom\mylabel{hdicht}\eeq
as well as $E\in\Hqonmom$ for some $E\in\Hqmom$\,, if and only if its extension by zero into $\tilde{\om}$ is an element
of $\Hqm(\tilde{\om})$ for any open set $\tilde{\om}$ with $\om\subset\omq\subset\tilde{\om}\subset\ol{\tilde{\om}}\subset M$\,.
The first assertion may be proved analogously to \cite[Theorem 3.6]{wloka} or \cite[Theorem 2.1]{agmon}
and the second analogously to \cite[Theorem 3.7]{wloka}. The same techniques yield finally
\beq\cqu(\omq)\qqtext{is dense in}\Rqom{}\qqtext{resp.}\Dqom{}\qquad.\mylabel{lemm0}\eeq
Especially for $\om=M=\rN$ we also have for all $s\in\rz$ that
$\cqun(\rN)$ is dense in $\Rqs(\rN)$\,, $\rqs(\rN)$\,, $\Dqs(\rN)$\,, $\dqs(\rN)$\,, $\Rqs(\rN)\cap\Dqs(\rN)$\,, 
$\rqs(\rN)\cap\dqs(\rN)$\,, $\Hqms(\rN)$ and $\hqms(\rN)$\,.

\subsection{Hodge-Helmholtz decompositions}\mylabel{apphodgedeco}

By the projection theorem, the $\Lzqom{}$-orthogonality of
$\rot\pR{q-1}{}{}{\circ}(\Omega)$ and $\dqnom{}$ as well as $\pdiv\Dqpeom{}$ and $\ronqnom{}$ and
the obvious inclusions $\rot\pR{q-1}{}{}{\circ}(\Omega)\subset\xrnqn\vono$ 
as well as $\pdiv\Dqpeom{}\subset\xdqn\vono$ we get the following
Hodge-Helmholtz decompositions \big(For details please see \cite[Lemma 1]{potential}, \cite[Lemma 1]{decomposition}
or in the classical case \cite[p. 168]{boundaryelectro}, \cite[Lemma 3.13]{xmas}.\big):

\begin{lem}\mylabel{lzzerlegungen}
The following $\skp{\eps\,\cdot\,}{\,\cdot\,}_{\lzqom}$-orthogonal (denoted by $\oplus_\eps$) decompositions hold for
admissible transformations $\eps$\,:
\begin{align*}
\text{\rm\bf{ (i)}}&&\lzqom
&=\ol{\rot\pR{q-1}{}{}{\circ}(\Omega)}\oplus_\eps\eps^\me\dqnom{}=\ronqnom{}\oplus_\eps\eps^\me\ol{\pdiv\Dqpeom{}}\\
&&&=\eps^\me\ol{\rot\pR{q-1}{}{}{\circ}(\Omega)}\oplus_\eps\dqnom{}=\eps^\me\ronqnom{}\oplus_\eps\ol{\pdiv\Dqpeom{}}\\
\text{\rm\bf{ (ii)}}&&\lzqom
&=\ol{\rot\pR{q-1}{}{}{\circ}(\Omega)}\oplus_\eps\dhqepsom{}\oplus_\eps\eps^\me\ol{\pdiv\Dqpeom{}}\\
&&&=\eps^\me\ol{\rot\pR{q-1}{}{}{\circ}(\Omega)}\oplus_\eps\eps^\me\dH{q}{}{\eps^\me}(\Omega)\oplus_\eps\ol{\pdiv\Dqpeom{}}
\end{align*}
All closures are taken in $\lzqom$\,.
\end{lem}

Here we introduced the `{\it (harmonic) Dirichlet forms}' by
$$\dhqepsom{}:=\xrnqn\vono\cap\eps^\me\xdqn\vono$$
and we denote them by $\dhqom$\,, if $\eps=\id$\,.
An easy application of the latter lemma shows that the orthogonal projection
$$\pi\,:\,\dH{q}{}{\nu}(\Omega)\To\dhqepsom{}$$
on $\eps^\me\xdqn\vono$ along $\rot\xrnqm\vono$ is well defined, linear, continuous and injective.
Therefore, by symmetry we obtain $\dim\dH{q}{}{\nu}(\Omega)=\dim\dhqepsom{}$ and hence this dimension is
independent of transformations, i.e.
$$\dim\dhqepsom{}=\dim\dhqom\qquad.$$

\subsection{Compact embedding}\mylabel{appcomp}

\begin{defini}\mylabel{maxkompakt}
$\Omega$ possesses the
\begin{itemize}
\item[\bf(i)]  `{\sf Maxwell compactness property}' ({\sf MCP}), if and only if the embeddings
$$\Ronqom{}\cap\Dqom{}\hookrightarrow\lzqom\qquad;$$
\item[\bf(ii)] `{\sf Maxwell local compactness property}' ({\sf MLCP}), if and only if the embeddings
$$\Ronqom{}\cap\Dqom{}\hookrightarrow\Lzqloc(\omq)$$
\end{itemize}
are compact for all $q$\,. 
\end{defini}

The {\sf MCP} and {\sf MLCP} are a properties of the boundary. We will shortly present some results.

There exists a large amount of literature about the {\sf MCP}, which can only hold for sub-manifolds $\om$ with compact closure,
which may be assumed by now.
The first idea was to use Gaffney's inequality, i.e. to estimate the $\qh{1}{q}{}{}\vono$-norm
by the $\xrq\vono\cap\xdq\vono$-norm, and then to apply Rellich's selection theorem.
To do this one needs smooth boundaries, which for instance may be seen in \cite[Theorem 8.6]{leisbuch}.
If $q=0$ we even have
$$\xr{0}{}{\circ}\vono\cap\xd{0}{}{}\vono=\xr{0}{}{\circ}\vono=\qh{1}{0}{}{\circ}\vono\qquad.$$
In 1972 Weck \cite{weckhabil} resp. \cite{weckmax} presented for the first time a proof of the 
{\sf MCP} for manifolds with nonsmooth boundaries (`cone-property'). 
Further proofs of the {\sf MCP} were given by Picard \cite{comimb} (`Lipschitz-domains') and in the classical
case by Weber \cite{weber} (another `cone-property') and Witsch \cite{witsch} (`$p$-cusp-property').
A proof of the {\sf MCP} in the classical case for bounded domains handling the largest known class of boundaries was given by
Picard, Weck and Witsch in \cite{xmas}. They combine the techniques from \cite{weckmax}, \cite{comimb} and \cite{witsch}.

We note that the {\sf MCP} is independent of transformations. More precisely:
Let $\eps_q$ be admissible transformations for all $q$\,.
Then $\Omega$ possesses the {\sf MCP}, if and only if the embeddings
$$\xrnq\vono\cap\eps_q^\me\xdq\vono\hookrightarrow\lzqom$$
are compact for all $q$\,.
Moreover, the {\sf MCP} yields the finite dimension of the space of Dirichlet forms $\dhqom{}$\,.
In fact, the dimension is determined by topological properties of $\Omega$\,,
i.e. $\dim\dhqom=\beta_{N-q}$\,, the $(N-q)$-th Betti number of $\Omega$\,. 
Furthermore, for admissible transformations the {\sf MCP} implies (by an indirect argument) 
the existence of a positive constant $c$\,, such that the estimate
\beq\norm{E}_{\lzqom}\leq c\big(\norm{\rot E}_{\lzqpeom}+\norm{\pdiv\eps E}_{\qLzom{q-1}{}}\big)\mylabel{abschaetzrotunddiv}\eeq
holds uniformly with respect to $E\in\xrnq\vono\cap\eps^\me\xdq\vono\cap\dhqepsom{}^\bot$\,.
Here we denote the orthogonality with respect to the $\skp{\,\cdot\,}{\,\cdot\,}_{\lzqom}$-scalar product by $\bot$\,.
This estimate implies the closedness of $\rot\xrnq\vono$ resp. $\pdiv\xdq\vono$ in $\qLzom{q+1}{}$ resp. $\qLzom{q-1}{}$\,.
We even have
\begin{align}
\ol{\rot\xrnq\vono}=\rot\xrnq\vono&=\rot\big(\xrnq\vono\cap\eps^\me\xdqn\vono\cap\dhqepsom{}^{\bot_\nu}\big)\qquad,
\mylabel{rotabg}\\
\ol{\pdiv\xdq\vono}=\pdiv\xdq\vono&=\pdiv\big(\xdq\vono\cap\eps^\me\xrnqn\vono\cap\dH{q}{}{\eps^\me}(\Omega)^{\bot_\nu}\big)\qquad,
\mylabel{divabg}
\end{align}
which was shown in \cite{potential} in the case $\eps=\nu=\id$\,. Here we denote the orthogonality with respect to
the $\skp{\nu\,\cdot\,}{\,\cdot\,}_{\lzqom}$-scalar product by $\bot_\nu$ and put $\bot:=\bot_{\id}$\,.

For an exterior domain $\om$ with the {\sf MLCP} we have similar results.
We will present them in the following.

\begin{lem}\mylabel{maxkompaktlokallem}
The following assertions are equivalent:
\begin{itemize}
\item[\rm\bf (i)] $\om$ possesses the {\sf MLCP}.
\item[\rm\bf (ii)] $\om\cap U_t$ possesses the {\sf MCP} for all $r\geq r_0$ with $\rN\ohne\om\subset U_{r_{0}}$\,.
\item[\rm\bf (iii)] The embeddings
$$\Ronqsom\cap\Dqsom\hookrightarrow\Lzqtom$$
are compact for all $t,s\in\rz$ with $t<s$ and all $q$\,.
\item[\rm\bf (iv)] For all $t,s\in\rz$ with $t<s$\,, all $q$ and all admissible transformations $\eps_q$ the embeddings
$$\Ronqsom\cap\eps_q^\me\Dqsom\hookrightarrow\Lzqtom$$
are compact.
\end{itemize}
\end{lem}

From \cite{potential} and \cite{hodge} we obtain $\dim\dhqom=\dim\dhqmeom=\beta_{N-q}<\infty$\,.
Here we introduced the `{\it (weighted harmonic) Dirichlet forms}'
$$\dhqepsom{t}:=\ronqnom{t}\cap\eps^\me\dqnom{t}$$
and again neglect the transformation or the weight in the notation for $\eps=\id$ or $t=0$\,.

Now let $\eps$ be an admissible transformation, which is $\tau$-$\pc{1}$-admissible of second kind 
in $A_{r}$ for an arbitrary $r\geq r_{0}$  
with some order of decay $\tau>0$ (and $r_{0}$ from Lemma \ref{maxkompaktlokallem}).

We need a fundamental Poincare-like estimate:

\begin{lem}\mylabel{rotdivunglepsomega}
There exists some constant $c>0$ and some compact set $K\subset\rN$\,,
such that
$$\norm{E}_{\Lzqom{-1}}\leq c\big(\norm{\rot E}_{\lzqpeom}+\norm{\pdiv\eps E}_{\qLzom{q-1}{}}+\norm{E}_{\lzq(\Omega\cap K)}\big)$$
holds true for all $E\in\rqom{-1}\cap\eps^\me\dqom{-1}$\,.
\end{lem}

\begin{proof}
By a usual cutting technique we may restrict our considerations to the special case $\Omega=\rN$
and $\eps$ is $\tau$-$\pc{1}$-admissible of second kind in $\rN$\,. 
Picking some $E$ from $\prq{-1}(\rN)\cap\eps^\me\pdq{-1}(\rN)$ by Lemma \ref{wholespaceregularity} (ii) we get
$E\in\qh{1}{q}{-t}{}(\rN)$ for all $t\geq1$ and the estimate (with $c$ depending on $t$ but not on $E$)
\beq\norm{E}_{\qh{1}{q}{-t}{}(\rN)}\leq c\big(\norm{E}_{\Lzq{-t}(\rN)}+\norm{\rot E}_{\Lzqpe{1-t}(\rN)}
+\norm{\pdiv\eps E}_{\qLz{q-1}{1-t}(\rN)}\big)\qquad.\mylabel{EinsNormabsch}\eeq
From \big[\cite{potential}, Lemma 5\big] we receive a compact set $K$\,, such that
$$\norm{E}_{\Lzq{-1}(\rN)}\leq c\big(\norm{\rot E}_{\lzqpe(\rN)}+\norm{\pdiv E}_{\qLz{q-1}{}(\rN)}
+\norm{E}_{\lzq(K)}\big)\qquad.$$
Then \eqref{EinsNormabsch} (for $t=1$) and the latter estimate yield with $\id=\eps-\epsd$
$$\norm{E}_{\qh{1}{q}{-1}{}(\rN)}\leq c\big(\norm{\rot E}_{\lzqpe(\rN)}+\norm{\pdiv\eps E}_{\qLz{q-1}{}(\rN)}
+\norm{E}_{\lzq(K)}+\norm{E}_{\qh{1}{q}{-1-\tau}{}(\rN)}\big)\quad.$$
Again utilizing \eqref{EinsNormabsch} (for $t=1+\tau$) the term $\norm{E}_{\qh{1}{q}{-1-\tau}{}(\rN)}$
may be replaced by $\norm{E}_{\Lzq{-1-\tau}(\rN)}$\,. 
Since $\tau>0$ and using the trick from \eqref{trick} this one can be swallowed by the left hand side, 
which might produce some other compact set $\tilde{K}\supset K$\,.
\end{proof}

We note that we did not need the {\sf MLCP} for the proof of this lemma.
But this lemma and the {\sf MLCP} yield directly (by an indirect argument)

\begin{cor}\mylabel{PoincareAbsch}
$\dhqepsom{-1}$ is finite dimensional and there exists some positive constant $c$\,, such that
$$\norm{E}_{\Lzqom{-1}}\leq c\big(\norm{\rot E}_{\lzqpeom}+\norm{\pdiv\eps E}_{\qLzom{q-1}{}}\big)$$
holds for all $E\in\ronqom{-1}\cap\eps^\me\dqom{-1}\cap\dhqepsom{-1}^{\bot_{-1,\nu}}$\,.
Here we denote by $\bot_{-1,\nu}$ the orthogonality with respect to the
$\skpom{\nu\rho^\me\,\cdot\,}{\rho^\me\,\cdot\,}$-scalar product.
\end{cor}

\begin{cor}\mylabel{abgeschlRaeume}
With closures taken in $\qLzom{q\pm1}{}$ we have
\begin{align*}
\text{\rm\bf{(i)}}\quad&&\ol{\rot\Ronqom{}}
&=\rot\ronqom{-1}=\rot\big(\ronqom{-1}\cap\eps^\me\dqnom{-1}\cap\dhqepsom{-1}^{\bot_{-1,\nu}}\big)\quad,\\
\text{\rm\bf{(ii)}}\quad&&\ol{\pdiv\Dqom{}}
&=\pdiv\dqom{-1}=\pdiv\big(\dqom{-1}\cap\eps^\me\ronqnom{-1}\cap\dH{q}{-1}{\eps^\me}(\Omega)^{\bot_{-1,\nu}}\big)\quad.
\end{align*}
\end{cor}

\begin{proof}
The proof is analogous to the one of \cite[Lemma 7]{potential}. Nevertheless, let us briefly
indicate how to prove {\bf (i)}. The other assertion follows similarly.
Let $(E_n)_{n\in\nz}\subset\Ronqom{}$ be some sequence with
$\rot E_n\xrightarrow{n\to\infty}G$ in $\Lzqpeom{}$\,. Using Lemma \ref{lzzerlegungen} we may assume
without loss of generality $E_n\in\Ronqom{}\cap\eps^\me\Dqnom{}$\,.
Moreover, by the projection theorem applied in $\Lzqom{-1}$ we may further assume
$$E_n\in\ronqom{-1}\cap\eps^\me\dqnom{-1}\cap\dhqepsom{-1}^{\bot_{-1,\nu}}\qquad.$$
By Corollary \ref{PoincareAbsch} $(E_n)_{n\in\nz}$ is a $\Lzqom{-1}$-Cauchy sequence
and the limit $E\in\Lzqom{-1}$ even is an element of
$\ronqom{-1}\cap\eps^\me\dqnom{-1}\cap\dhqepsom{-1}^{\bot_{-1,\nu}}$\,,
which completes the proof.
\end{proof}

Finally we note an immediate and easy conclusion of Corollary \ref{abgeschlRaeume}, 
i.e. an electro-magneto static solution theory handling homogeneous tangential boundary data.

\begin{theo}\mylabel{loesstatikgewichtNull}
Let $d^q:=\dim\dhqepsom{-1}$ continuous linear functionals $\Phi^\ell_\eps$ 
on the space $\rqom{-1}\cap\eps^\me\dqom{-1}$ with
$$\dhqepsom{-1}\cap\bigcap_{\ell=1}^{d^q}N(\Phi^\ell_\eps)=\{0\}$$
be given. Then with $\Phi_\eps:=(\Phi^1_\eps\,\,,\dots,\Phi^{d^q}_\eps\,\,)$ the linear operator
$$\Abb{\MAX_\eps}{\ronqom{-1}\cap\eps^\me\dqom{-1}}{\ol{\pdiv\Dqom{}}\times\ol{\rot\Ronqom{}}\times\cz^{d^q}}{E}{\big(\pdiv\eps E,\rot E,\Phi_\eps(E)\big)}$$
is a topological isomorphism.
Here $N(\Phi^\ell_\eps)$ denotes the kernel of $\Phi^\ell_\eps$\,.
\end{theo}

\subsection{Linear transformations}

Elementary calculations and estimates show for open subsets $\om$\,, $\tilde{\om}$ with compact closure
of smooth Riemannian manifolds $M$\,, $\tilde{M}$\,:

\begin{lem}\mylabel{taudivrot}
Let $\tau:\tilde{\om}\to\Omega$ be a $\xsc{2}{}$-diffeomorphism respecting orientation
and $\eps$ be a linear transformation. Then the transformation $\eps$ is admissible,
if and only if the transformation
$$\eps_\tau:=(-1)^{q(N-q)}*\tau^**\eps(\tau^*)^\me$$
is admissible. In particular $\id_\tau=(-1)^{q(N-q)}*\tau^**(\tau^*)^\me$ is admissible. Furthermore:
\begin{itemize}
\item[\bf(i)] $E\in\pr{q}{}{}{}(\om)$ resp. $E\in\pr{q}{}{}{\circ}(\om)$\,,
if and only if $\tau^*E\in\pr{q}{}{}{}(\tilde{\om})$ resp. $\tau^*E\in\pr{q}{}{}{\circ}(\tilde{\om})$\,.
Moreover, $\xrot\tau^*E=\tau^*\xrot E$ and there exists a constant $c>0$ independent of $E$\,, such that
$$c^\me\normu{E}{\xrq\vono}\leq\normu{\tau^*E}{\xrq(\tilde{\om})}\leq c\normu{E}{\xrq\vono}\qquad.$$
\item[\bf(ii)] $E\in\eps^{-1}\pdi{q}{}{}{}(\om)$\,, if and only if
$\tau^*E\in\eps_\tau^{-1}\pdi{q}{}{}{}(\tilde{\om})$\,.
Moreover, $\xdiv\eps_\tau\tau^*E=\id_\tau\tau^*\xdiv\eps E$ holds
and there exists some $c>0$ independent of $E$ or $\eps_\tau$\,, such that
$$c^\me\normu{E}{\eps^{-1}\xdq\vono}\leq\normu{\tau^*E}{\eps_\tau^{-1}\xdq(\tilde{\om})}\leq c\normu{E}{\eps^{-1}\xdq\vono}\qquad.$$
\end{itemize}
\end{lem}

\subsection{Fourier transformation for differential forms}

In the special case $M=\rN$ we have some useful operators from the spherical calculus developed in \cite{sphharm}.
For Euclidean coordinates $\{x_1,\dots,x_N\}$ we introduce the pointwise linear operators $R$\,, $T$ on $q$-forms by
$$RE:=x_n\pd x^n\wedge E=r\pd r\wedge E\qqtext{,}TE:=(-1)^{(q-1)N}*R*E$$
and recall the formulas
\begin{equation}RR=0\qqtext{,}TT=0\qqtext{,}RT+TR=r^2\mylabel{RTformula}\end{equation}
as well as for $q$-forms $E$ and $(q+1)$-forms $H$
\beq RE\wedge*H=E\wedge*TH\qqtext{,}TH\wedge*E=H\wedge*RE\qquad,\mylabel{RTskpformula}\eeq
i.e. $\skp{RE}{H}_{q+1}=\skp{E}{TH}_q$ using the pointwise scalar product for differential forms.
Then, for example, the differential $\rot$ resp. $\pdiv$ corresponds to $R$ resp. $T$ in the sense that
\beq C_{\rot\,,\varphi(r)}E=\varphi'(r)r^\me RE\qqtext{resp.}C_{\pdiv,\varphi(r)}E=\varphi'(r)r^\me TE\mylabel{RTcomformula}\eeq
holds for $\varphi\in\pc{1}(\rz)$ and $E\in\pRq{}(\rN)$ resp. $E\in\pDq{}(\rN)$\,.

But there is at least one more connection between these operators.
Let us present the componentwise (with respect to Euclidean coordinates) Fourier transformation on $q$-forms $\calF$\,,
which is a unitary mapping on $\lzq(\rN)$\,. With $\mathcal{X}(x):=x$ and the well known formula         
$$\calF(\pa u)=\ie^{|\alpha|}\mathcal{X}^\alpha\calF(u)$$
for scalar distributions $u$ we get some formulas for $\calF$ operating on $q$-forms $E$\,:
\begin{align}
\calF*E&=*\calF E\mylabel{fouriereins}\\
\calF(\pa E)&=\ie^{|\alpha|}\mathcal{X}^\alpha\calF(E)&,&&\pa\calF(E)&=(-\ie)^{|\alpha|}\calF(\mathcal{X}^\alpha E)\mylabel{fourierzwei}\\
\calF(\rot E)&=\ie R\calF(E)&,&&\rot\calF(E)&=-\ie\calF(RE)\mylabel{fourierdrei}\\
\calF(\pdiv E)&=\ie T\calF(E)&,&&\pdiv\calF(E)&=-\ie\calF(TE)\mylabel{fouriervier}\\
\calF(\Delta E)&=-r^2\calF(E)&,&&\Delta\calF(E)&=-\calF(r^2 E)\mylabel{fourierfuenf}
\end{align}
These formulas may be checked for smooth forms from Schwartz' space and hence remain valid for distributional $q$-forms, i.e.
extend to our weak calculus. We note $\rot\pdiv+\pdiv\rot=\Delta$\,,
where the Laplacian $\Delta$ acts on each Euclidean component of $E$\,.
Utilizing these formulas some Sobolev spaces can be characterized with the aid of the Fourier transformation.
We easily get:
\begin{align}
\Hqm(\rN)&=\setb{E\in\Lzq{}(\rN)}{\calF(E)\in\Lzq{m}(\rN)}\qquad,\qquad m\in\nz\mylabel{fouriersieben}\\
\pRq{}(\rN)&=\setb{E\in\Lzq{}(\rN)}{R\calF(E)\in\Lzqpe{}(\rN)}\mylabel{fourieracht}\\
\pDq{}(\rN)&=\setb{E\in\Lzq{}(\rN)}{T\calF(E)\in\qLz{q-1}{}(\rN)}\mylabel{fourierneun}
\end{align}
In this sense we also may define $\qH{s}{q}{}{}(\rN)$\,, if $s\in\rz$\,.

\subsection{Some technical lemmas}

\begin{lem}\mylabel{spiegel}
Let $r>0$\,, $x':=(x_1,\cdots,x_{N-1})$ and
$$\Abb{\tau}{U_r^+}{U_r^-}{x}{(x',-x_N)}\qquad.$$
Then the mirror operator
$$S_{\rot}:\xrq(U_r^-)\to\xrq(U_r)$$
defined by $\restr{S_{\rot}\,E}{U_r^-}:=E$ and $\restr{S_{\rot}\,E}{U_r^+}:=\tau^*E$
is well defined, linear and continuous.
$S_{\rot}$ commutates with $\rot$ and $\normu{S_{\rot}\,E}{\xlq(U_r)}=\sqrt{2}\normu{E}{\xlq(U_r^-)}$ holds.
\big($\sqrt2/2 S_{\rot}$ even is an isometry.\big)
Moreover, if $\supp E\subset\ol{U_\varrho^-}$ for some $\varrho<r$\,, then $\supp S_{\rot}\,E\subset\ol{U_\varrho}$\,.

The dual mirror operator
\beq S_{\pdiv}:=(-1)^{q(N-q)}*S_{\rot}\,*:\xdq(U_r^-)\to\xdq(U_r)\mylabel{divspiegel}\eeq
has the corresponding properties.
\end{lem}

\begin{proof}
By density it is enough to show $S_{\rot}\,E\in\xrq(U_r)$ and $\rot S_{\rot}\,E=S_{\rot}\rot E$ 
for some $E\in\xciq(\ol{U_r^-})$\,.
The assertions about the continuity and the support follow directly.
Let $\iota:U_r^0\hookrightarrow\ol{U_r^-}$ denote the natural embedding.
Observing that $\tau$ changes the orientation we get from Stokes' theorem for $\Phi\in\xcinqp(U_r)$
(Clearly we identify $\Phi$ with its restrictions on $U_r^\pm$\,.)
\begin{align*}
\spu{S_{\rot}\,E}{\xdiv\Phi}{\xlq(U_r)}
&=(-1)^{q^2}\int_{U_r^-}E\wedge(\d*\bar{\Phi})+(-1)^{q^2}\int_{U_r^+}(\tau^*E)\wedge(\d*\bar{\Phi})\\
&=(-1)^q\int_{U_r^-}E\wedge\d\big(*\bar{\Phi}-(\tau^{-1})^**\bar{\Phi}\big)\\
&=-\int_{U_r^-}(\d E)\wedge\big(*\bar{\Phi}-(\tau^{-1})^**\bar{\Phi}\big)\\
&\qquad\qquad+\int_{U_r^0}(\iota^*E)\wedge\Big(\big(\iota^*-\iota^*(\tau^{-1})^*\big)*\bar{\Phi}\Big)\qquad.
\intertext{By $\iota-\tau^{-1}\circ\iota=0$ the boundary integral vanishes and we obtain}
\spu{S_{\rot}\,E}{\xdiv\Phi}{\xlq(U_r)}
&=-\int_{U_r^-}(\d E)\wedge*\bar{\Phi}-\int_{U_r^+}(\tau^*\d E)\wedge*\bar{\Phi}\\
&=-\spu{S_{\rot}\rot E}{\Phi}{\xlqp(U_r)}\qquad.
\end{align*}
\end{proof}

\begin{lem}\mylabel{fourier}
Let $N\ge3$ and $\varrho>0$\,. There exists a constant $c>0$\,, such that for all $E\in\xdqn(\rN)$ with $\supp E\subset U_\varrho$
there exists some $H\in\xh{1}{q+1}{}(\rN)$ satisfying
$$\xdiv H=E\qqtext{,}\normu{H}{\xh{1}{q+1}{}(\rN)}\le c\normu{E}{\xlq(\rN)}\qquad.$$
\end{lem}

\begin{proof}
Let $E\in\xdqn(\rN)$ with $\supp E\subset U_\varrho$\,. By Fourier's transformation we get for the Euclidian components of 
$E=E_{I}\pd x^{I}$
$$\calF E_{I}\in\pc{0}(\rN)\qqtext{,}
\forall\,x\in\rN\quad\big|\calF E_I(x)\big|\le \lambda(U_\varrho)^{1/2}\normu{E}{\xlq(\rN)}\qquad,$$
where $\lambda$ denotes Lebesgue's measure. Hence, all components of $\calF E$ are bounded. 
Let $\hat{H}:=r^{-2}R\calF E$ with $\hat{H}(0):=0$\,. The estimate
$$\big|\hat{H}_J(x)\big|\le c|x|^{-1}\sum_{I}\big|\calF E_I(x)\big|$$
holds for all $x\in\rN\ohne\{0\}$ and all indices $J$ and implies $\mathcal{X}_n\hat{H}\in\lzqpe(\rN)$\,.
Hence, $\hat{H},\calF^{-1}\hat{H}\in\xlqp(\rN)$ since $N\ge3$\,. Moreover, we get
$$\norm{\hat{H}}_{\lzqpe(\rN)}+\norm{r\hat{H}}_{\lzqpe(\rN)}\leq c\norm{E}_{\lzq(\rN)}\qquad.$$
Thus by \eqref{fouriersieben} 
$H:=-\ie\calF^\me\hat{H}\in\xh{1}{q+1}{}(\rN)$ with $\normu{H}{\xh{1}{q+1}{}(\rN)}\le c\normu{E}{\xlq(\rN)}$\,.
Using \eqref{fouriervier} as well as \eqref{RTformula} we finally obtain
$$\pdiv H=\calF^\me T\hat{H}=\calF^\me r^{-2}TR\calF E=E$$
because $\pdiv E=0$ yields $T\calF E=0$ again by \eqref{fouriervier}.
\end{proof}

To prepare the final lemma of the appendix let $U\subset\rN$ and
$$\Phi=\sum_{I}\Phi_I\dx^{I}\in \xlq(U)\qquad.$$
Then $\Phi=\Phi^\tau+\Phi^\rho$ is an (pointwise) orthogonal decomposition, where
\beq\Phi^\tau:=\sum_{N\not\in I}\Phi_I\dx^{I}\qqtext{,}
\Phi^\rho:=\sum_{N\in I}\Phi_I\dx^I\qquad.\mylabel{taurhodecodef}\eeq

\begin{lem}\mylabel{hmschluss}
Let $U\subset\rN$\,, $m\in\nz$\,, $E\in\lzq(U)$ and $\eps$ be a $\pc{m}$-admissible transformation.
Furthermore, let $E^\tau$ and $(\eps E)^\rho$ be elements of $\xh{m}{q}{}(U)$\,. 
Then $E$ belongs to $\xh{m}{q}{}(U)$ as well.
\end{lem}

\begin{proof}
We have $(\eps E^{\rho})^{\rho}=(\eps E)^{\rho}-(\eps E^{\tau})^{\rho}\in\xh{m}{q}{}(U)$\,.
Since the restriction $\eps^{\rho,\rho}$ of $\eps$ acting on the normal parts, 
i.e. $\eps^{{}\rho,\rho}E^\rho=(\eps E^{\rho})^{\rho}$\,,
is pointwise invertible with $\xsc{m}{}(\ol{U})$ entries we obtain $E^\rho\in\xh{m}{q}{}(U)$\,.
\end{proof}

\section{Translation to the classical electro-magnetic language}

Finally we present our results in the classical language of vector analysis, i.e. $M=\rd$\,.
By the usual identifications we have to following table:
\begin{center}
\fbox{\begin{tabular}{|c||c|c|c|c|}
\hline
& $q=0$ & $q=1$ & $q=2$ & $q=3$\\
\hline\hline
$\rot$ & $\grad$ & $\curl$ & $\pkdiv$ & $0$\\
\hline
$\pdiv$ & $0$ & $\pkdiv$ & $-\curl$ & $\grad$\\
\hline
\end{tabular}}
\end{center}

Now we deal with the usual Sobolev spaces $\hm(\om)$ and $\h{}{}{}(\curl,\om)$\,, $\h{}{}{}(\pkdiv,\om)$
as well as the trace spaces $\h{m-1/2}{}{}(\dom)$\,. Moreover, we have the weighted Sobolev spaces
$\hms(\om)$\,, $\Hms(\om)$ as well as for $\Diamond\in\{\curl,\pkdiv\}$
$$\H{}{s}{}(\Diamond,\om):=\setb{E\in\Lzsom}{\Diamond E\in\Lzsom}\qquad.$$

\begin{center}
\fbox{\begin{tabular}{|c||c|c|c|c|}
\hline
& $q=0$ & $q=1$ & $q=2$ & $q=3$\\
\hline\hline
$\Rqsom$ & $\H{1}{s}{}(\om)$ & $\H{}{s}{}(\curl,\om)$ & $\H{}{s}{}(\pkdiv,\om)$ & $\Lzsom$\\
\hline
$\Dqsom$ & $\Lzsom$ & $\H{}{s}{}(\pkdiv,\om)$ & $\H{}{s}{}(\curl,\om)$ & $\H{1}{s}{}(\om)$\\
\hline
\end{tabular}}
\end{center}

\begin{theo}\mylabel{regularitybdclass}
Let $m\in\nzn$\,, $\om$ be a bounded domain in $\rd$ with $\pc{m+2}$-boundary
as well as $\eps$ be some $\pc{m+1}$-admissible matrix. Furthermore, let
$$E\in\h{}{}{}(\curl,\om)\cap\eps^\me\h{}{}{}(\pkdiv,\om)$$
with
$$\curl E\in\hmom\qtext{,}\pkdiv\eps E\in\hmom\qtext{,}\nu\times E\in\h{m+1/2}{}{}(\dom)\quad.$$
Then $E\in\hom{m+1}{}{}$ and there exists a positive constant $c$ independent of $E$\,, such that
\begin{align*}
&\qquad\norm{E}_{\hom{m+1}{}{}}\\
&\leq c\big(\norm{E}_{\lzom}+\norm{\curl E}_{\hmom}+\norm{\pkdiv\eps E}_{\hmom}
+\norm{\nu\times E}_{\h{m+1/2}{}{}(\dom)}\big)\quad.
\end{align*}
\end{theo}

This theorem may be regarded as a generalization to inhomogeneous boundary data of \cite{weberreg},
whereas the next theorem represents a new result even in the classical context.

\begin{theo}\mylabel{regularityunbdclass}
Let $s\in\rz$\,, $m\in\nzn$\,, $\Omega\subset\rd$ be an exterior domain with $\pc{m+2}$-boundary
and $\eps$ be some $\pc{m+1}$-admissible matrix. Furthermore, let
$$E\in\H{}{s}{}(\curl,\om)\cap\eps^\me\H{}{s}{}(\pkdiv,\om)\qqtext{with}\nu\times E\in\h{m+1/2}{}{}(\dom)\qquad.$$
\begin{itemize}
\item[\rm\bf (i)] Then $\curl E\in\Hmsom$ and $\pkdiv\eps E\in\Hmsom$ imply $E\in\Hom{m+1}{s}{}$ and with some constant $c>0$
\begin{align*}
&\qquad\norm{E}_{\Hom{m+1}{s}{}}\\
&\leq c\big(\norm{E}_{\Lzsom}+\norm{\curl E}_{\Hmsom}
+\norm{\pkdiv\eps E}_{\Hmsom}+\norm{\nu\times E}_{\h{m+1/2}{}{}(\dom)}\big)
\end{align*}
holds uniformly with respect to $E$\,.
\item[\rm\bf (ii)] If additionally $\eps$ is $0$-$\pc{m+1}$-admissible of second kind and $\tau$-$\pc{0}$-admissible
of first (or second) kind with some $\tau>0$
then $\curl E\in\hom{m}{s+1}{}$ and $\pkdiv\eps E\in\hom{m}{s+1}{}$ imply $E\in\hom{m+1}{s}{}$
and there exists some positive constant $c$\,, such that the estimate
\begin{align*}
&\qquad\norm{E}_{\hom{m+1}{s}{}}\\
&\leq c\big(\norm{E}_{\Lzsom}+\norm{\curl E}_{\hom{m}{s+1}{}}
+\norm{\pkdiv\eps E}_{\hom{m}{s+1}{}}+\norm{\nu\times E}_{\h{m+1/2}{}{}(\dom)}\big)
\end{align*}
holds uniformly with respect to $E$\,.
\end{itemize}
\end{theo}

Here we denoted by $\nu$ the exterior normal unit vector at $\dom$\,.

\begin{rem}\mylabel{regularityremclass}
Similar results hold for kinds of spaces like $\eps^\me\h{}{}{}(\curl,\om)\cap\h{}{}{}(\pkdiv,\om)$
and/or with prescribed normal traces $\nu\cdot E$ resp. $\nu\cdot\eps E$\,.
\end{rem}

\begin{acknow}
This research was supported by the {\it Deutsche Forschungsgemeinschaft}
via the project {\sf `We 2394: Untersuchungen der Spektralschar verallgemeinerter
Maxwell-Operatoren in unbeschr\"ankten Gebieten'} and by the {\it Department of Mathematical
Information Technology} of the University of Jyv\"askyl\"a, Finland, where the second author spent
a sabbatical of three month during spring 2007 as a postdoc fellow.

The authors are particularly indebted to their academic teachers Norbert Weck and Karl-Josef Witsch 
for introducing them to the field.
\end{acknow}

\end{document}